\documentclass[10pt,reqno]{amsart}
\usepackage{amsmath}
\usepackage{amssymb,amscd}
\usepackage[all]{xy}
\usepackage[english]{babel}

\usepackage{times}
\usepackage{t1enc}
\usepackage{graphicx}
\usepackage{float}
\usepackage{tikz}
\usepackage{tikz-cd}
\usepackage{pdfpages}
\usepackage{caption}
\usepackage{subcaption}
\usepackage{epstopdf}
\usepackage{epsfig}
\usepackage{mathrsfs}
\usepackage{setspace} 
\usepackage[mathscr]{euscript}
\usepackage[utf8]{inputenc}
\usepackage[pdftex]{hyperref} 
\usepackage{amsmath, amssymb, amsthm, amsfonts, amstext,  textcomp, bm, amscd}
\usepackage[all]{xy}
\usepackage[english]{babel}
\usepackage{xcolor,soul}
\usetikzlibrary{matrix,arrows}

\numberwithin{equation}{section}
\newtheorem{thm}{Theorem}[section]
\newtheorem*{thm*}{Theorem}
\newtheorem{prop}[thm]{Proposition}
\newtheorem{lem}[thm]{Lemma}

\newtheorem{rmk}{Remark}

\numberwithin{equation}{section}

\def\F{\mathbb{F}}
\def\Z{\mathbb{Z}}
\def\A{\mathcal{A}}

\def\H{\mathcal{H}}
\def\C{\mathbb{C}}

\def\u{N_{q}(M^{s}_{\mathcal O_X}(2,0))}

\def\bigcirc{O}
\def\w{N_q(M_{\text{L}}(n,d))}
\def\cM{\mathcal{M}_{L}(n,d)}
\def\M{M_{L}(n,d)}
\def\Mus{\mathcal{M}_{L}^{\text{us}}(n,d)}
\def\A{\frac{1}{N_{q}(\Aut E)}}
\def\fP{\mathfrak P}
\def \N{\widetilde{N}}

\def\hh{\H_{\gamma,q}}
\def\hhh{\H_{\gamma,q^{2}}}

\def \({\left(}
\def \){\right)}
\def \<{\langle}
\def \>{\rangle}
\def \bar{\overline}
\def \deg{\mathrm{deg}}

\def \ds{\displaystyle}

\def \Aut{{\rm Aut}}
\def \End{{\rm End}}

\def \Hom{{\rm Hom}}

\DeclareMathOperator{\Gal}{Gal}

\begin{document}
	
	\title[Statistics of Moduli Space]{Statistics of Moduli Space of vector bundles {\textbf{$\mathrm{II}$}}}

	\author[A. Dey]{Arijit Dey}
	\address{Department of Mathematics, Indian Institute of Technology-Madras, Chennai, India }
	
	\email{arijitdey@gmail.com}
	
	\author[S. Dey]{Sampa Dey}
	\address{Stat Math Unit,Indian Statistical Institute, Kolkata, India }
	
	\email{sampa.math@gmail.com}

	\author[A. Mukhopadhyay]{Anirban Mukhopadhyay} 
	\address{Institute of Mathematical Sciences, HBNI, CIT Campus, Taramani,  Chennai, India}
	
	\email{anirban@imsc.res.in}
	
	

	\begin{abstract} 
		
		Let $X$ be a smooth irreducible projective curve of genus $g \geq 2$ over a finite field $\F_{q}$ of characteristic $p$ with  $q$ elements  such that the function field $\F_{q}(X)$ is a geometric Galois extension of the rational function field of degree $N.$ Consider $gcd(n,d)=1$, let $M_{L}(n,d)$ be the moduli space of rank $n$ stable vector bundles over $X$ with fixed determinant isomorphic to a $\mathbb F_q$-rational line bundle $L$. Suppose $N_q (M_L(n,d))$ denotes the cardinality of the set of  $\F_{q}$-rational points of $M_{L}(n,d)$. We give an asymptotic bound of  $\log(N_{q}(M_{L}(n,d)) -  (n^2-1)(g-1)\log{q})$ for large genus $g,$ depending on $N$. Further, considering this logarithmic difference as a random variable, we prove a central limit theorem over a large family of hyperelliptic curves with uniform probability measure. Further, over the same family of hyperelliptic curves, we study the distribution of $\F_{q}$-rational points over the moduli space of rank $2$ stable vector bundles with trivial determinant  $M^{s}_{\mathcal{O}_{H}}(2,0)$ and it's Seshadri desingularisation ${\widetilde{N}}$ by choosing an appropriate random variable in each case. We also see that the corresponding random variables having standard Gaussian distribution as $g$ and $q$ tends to infinity.
\end{abstract}

\maketitle

\tableofcontents

\section{Introduction}

Let $V$ be a quasi-projective variety defined over a finite field $\F_{q}$ of characteristic $p$ with $q$ elements. Studying $\mathbb{F}_{q}$-rational points $N_q(V)$ of $V$ is significant in various branches of mathematics such as finite field theory, number theory, algebraic geometry and so on. For example, one of the important theorem in algebraic geometry  known as famous Weil conjecture \cite{We1} which allows one to compute Betti numbers of a smooth complex projective variety and $\mathbb F_{q}$-rational points of certain varieties associated to it.  This method was first exploitted by Harder and Narasimhan \cite{HaNa} and later by Desale and Ramanan \cite{DeRa} to find a comprehensive method to compute Betti numbers of  moduli space of bundles. 
Finding asymptotic formulas for $N_{q^m}(V)$ in terms of $q$ or other parameters related to $V,$ is another active area of research. In this paper our aim is to study the asymptotic behaviour of the $\mathbb{F}_{q}$-rational points of  moduli space of vector bundles on a smooth projective curve. 

Let $X$ be a smooth irreducible projective curve of genus $g \geq 2$ over a finite field $\F_{q}$ with $q$ elements of characteristic $p$ and $L$ be a line bundle on $X$ of degree $d,$ defined over $\mathbb F_q.$ Let $M(n,d)$ (resp, $M^s(n,d)$) be the moduli space of semistable (resp, stable) vector bundles of rank $n$ and degree $d,$ and $M_{L}(n,d)$ (resp. $M^{s}_{L}(n,d)$) be the moduli space of semistable (resp. stable) vector bundles of rank $n$ with fixed determinant $L.$ When $n$ and $d$ are coprime, the moduli space $M_{L}(n,d)$ is an irreducible smooth projective variety of dimension $(n^{2}-1)(g-1).$ Replacing $\F_q$ by a finite extension if necessary, we may assume that everything i.e. $X$, $L$ and $M_{L}(n,d)$ are defined over $\F_q.$ It is known that when $gcd(n,d)=1,$ the $\F_{q}$-rational points of $M_{L}(n,d)$ are precisely the isomorphism classes of stable vector bundles on $X$ defined over $\F_{q}$  [Proposition 1.2.1, \cite{HaNa}].
Now, when rank $n=1$ and degree $d=0,$ the moduli space $M(1,0)$ is the Jacobian $J_{X}$ of the curve $X,$ which is an abelian variety of dimension $g.$ Due to \textit{Weil conjectures}, the functional equation and analogue of the Riemann hypothesis for zeta function of a smooth projective curve of genus $g$ implies:
\[(\sqrt{q}-1)^{2g}\leq N_{q}(J_{X})\leq (\sqrt{q}+1)^{2g}.\]
For $g=1,$ this bound is tight due to the classical result of Deuring \cite{Deu}. For higher genus we know several improvements of this bound by Rosenbloom and Tsfasman\cite{RoTs}, Quebbemann\cite{Qu}, Tsfasman\cite{Ts}, Stein and Teske\cite{StTe} and others. For example, in \cite{Ts} Tsfasman has shown that for a fixed finite field $\F_{q},$
\[g\log{q} + o(g)\leq \log(N_{q}(J_{X}))\leq g\left( \log{q}+(\sqrt{q}-1)\log\frac{q}{q-1}\right)+o(g) \]
as $g\rightarrow \infty.$ In terms of \textit{gonality} (that is the smallest integer $d$ such that $X$ admits a non-constant map of degree $d$ to the projective line over $\F_{q}$), 
Shparlinski \cite{Sh}, showed that  
\[ \log(N_{q}(J_{X}))=g\left(\log{q}+O_{q}\left(\frac{1}{\log{\left( \frac{g}{d}\right) }} \right)  \right) \]
as $q$ is fixed and $g\rightarrow \infty.$
In \cite{XiZa}, when the function field $\F_{q}(X)$ is a geometric Galois extension over the field of rational functions $\F_{q}(x)$ of degree $N,$ Xiong and Zaharescu estimated $N_{q}(J_{X})$ in terms of $q,\,g$ and $N.$ They gave the following explicit bound [Theorem 1,  \cite{XiZa}]
\[
| \log(N_{q}(J_{X}))-g\log{q} | 
\leq
(N-1)
\left( 
\log \text{max} 
\left\lbrace 
1, \frac{       \log{    \(    \frac{7g}{N-1}      \)    }     }             {\log q}
\right\rbrace 
+3 
\right) ,
\]
which holds true for any $q$ and $g.$ Precisely, we see the quantity $\left( \log(N_{q}(J_{X}))-g\log{q}\right) $ is essentially bounded by $O\left(  \log \log g  \right), $ which is significantly smaller than the bound $O\left(  \frac{g}{\log g} \right) $ given by Shparlinski in \cite{Sh}.

Motivated by the work in \cite{XiZa} of Xiong and Zaharescu, we got interested in studying similar problems for moduli space of  stable rank $n$ and degree $d$ vector bundles with fixed determinant  on a smooth curve. This can be interpreted as a non-abelian analogue of their work. Now onwards, we will be considering a smooth projective curve $X$ of genus $g,$ defined over $\F_{q},$ where the function field $\F_{q}(X)$ is a geometric Galois extension over the field of rational functions $\F_{q}(x)$ of degree $N.$ We will simply call such curves as \emph{Galois curve} of degree $N.$ Here ``geometric '' means $\F_{q}$ is algebraically closed inside $\F_{q}(X).$
In a previous work , we explicitly studied the case when rank $n=2$ and degree $d=1,$ and we write the bound for $\log \left( N_{q}\left( M_{L}(2,1)\right) \right) $ in terms of $q,\,g$ and $N$ [Theorem 1.1, \cite{ASA}]. To prove this we used the Siegel formula\eqref{Si2.1}, where we need to count the number of isomorphism classes of unstable bundles too.  In this paper, we have generalized our previous results for the fixed determinant moduli space of any rank $n$ and degree $d$ with the condition that $gcd(n,d)=1.$ The challenging part was estimating the number of $\mathbb F_q$-rational points of  isomophism classes of unstable vector bundles. In this regard, the significant input comes from the work of Desale and Ramanan [Proposition 1.7, \cite{DeRa}], by which we could give an asymptotic bound in terms of $g,$ and $q$ inductively (cf. Proposition \ref{Pdom5.1}). It is worth mentioning here that one can do similar computations for non-coprime case too but the difficulty will arise in computing the number of $\mathbb F_q$-rational points 
of strictly-semistable strata.  In the rank $2$ and trivial determinant this was done in \cite{BaSe1} with great bit of care, and by using this we are able to do similar study for $M_{\mathcal O_X}(2,0)$ and it's Seshadri desingularization. \\

Our first main result is the following asymptotic formula for $\w$ in terms of $N, \, q$ and $g.$ 

\begin{thm}\label{thm1.3}
Let	$X$ be a Galois curve of degree $N$ of genus $g\geq 2$ over $\mathbb{F}_{q}.$ Assume that $n,$ and $d$ are coprime. 
If $\log{g} > \kappa\log(Nq)$ for some $\kappa>0$, sufficiently large absolute constant (independent of $N$), then for a constant $\sigma>0,$ depending only on $n,$
we have
	\begin{align*}
	\log(\w)&= (n^2-1)(g-1)\log{q}\,
	 + 
O\left( 
A+q^{-\sigma g}\exp(A)
\right), 
	\end{align*}
	where 
	$A=N\left( \frac{1}{\sqrt q}+\frac{\log\log g}{q^2}\right) ,$ and 
	the implied constant depends on $n,$ and $N.$
\end{thm}

%

Next we restrict our attention to the family of hyperelliptic curves i.e when $N=2$. Assume that $q$ is odd and  $\gamma$ is a positive integer $\geq\,5$.  Let $\H_{\gamma,q}$ be a family of curves given by the equation $y^2=F_{t}(x)$, where $F_{t}$ is a monic, square-free polynomial of degree $\gamma$ with coefficients in  $\mathbb{F}_{q}.$ Every such curve corresponds to an affine model of a unique projective hyperelliptic curve $H_t$, with genus $ g = \left[\frac{\gamma-1}{2}\right]$.

On $\H_{\gamma,q}$ we consider the probability space with the uniform probability measure obtained by picking the coefficients of $F_t(x)$ uniformly from the $\F_{q}$-vector space $\F_{q}^{\gamma}.$ With this set up, when $g$ is 
fixed and $q$ is growing, Katz and Sarnak showed that $ \sqrt{q}(\log N_{q}(J_{H_t})- g \log {q})$ is distributed as the 
trace of a random $2g\times2g$ unitary symplectic matrix \cite[Chapter 10, Variant 10.1.18]{KaSa}. On the other side, 
when the finite field is fixed and the genus $g$ grows, Xiong and Zaharescu \cite{XiZa} found  the limiting distribution 
of $\log N_{q}(J_{H_t})- g \log {q}$ 
in terms of its characteristic function. Moreover, when both $g$ and $q$ grow, they\cite{XiZa} showed that $ \sqrt{q}(\log N_{q}(J_{H_t})- g \log {q}) $ has a standard Gaussian distribution. Now, for every $H_t$ in $\H_{\gamma,q}$, fix a polarization i.e. a line bundle $L_t$ of degree $1$. In our previous work, we have studied the fluctuations of the quantity
$\log N_q (M_{L_t}(2,1)) - 3(g-1)\log q$ as the polarized curve $(H_t, L_t)$ varies over a large family of polarized hyperelliptic
curves. Since we are interested in $\mathbb F_{q}$-rational points of moduli spaces which is independent of the determinant  [Proposition 1.7, \cite{DeRa}], we denote the family of polarized curves by the same notation $\H_{\gamma,q}$. In this case, first we write the limiting distribution of $N_q (M_{L_t}(2,1)) -  3(g-1)\log{q}$ as $g$ grows and $q$ is fixed, in terms of it's characteristic function. Further, when $g$ and $q$ both grows together we see that $q^{3/2}\left(\log N_q (M_{L_t}(2,1)) -3(g-1)\log{q}\right)$ has a standard Gaussian distribution (see [Theorem 1.2, \cite{ASA}].) \\

Now, on the probability space $\H_{\gamma,q}$, we consider the random variable $\mathfrak{R}_{(g,q)},$
\[
\mathfrak R_{(g,q)} : \H_{\gamma,q} \rightarrow \mathbb{R}
\]
which sends any element $(H_t, L_t)$ in $\H_{\gamma, q}$ to the difference $\left( \log N_{q}(M_{L_t}(n,d))- (n^2-1)(g-1)\log{q}\right)$ in $\mathbb{R}.$ 
We find that, over the family $\H_{\gamma, q},$ the random variable  $\mathfrak{R}_{(g,q)}$ has a limiting distribution as $g$ grows and $q$ is fixed, we write it in terms of the characteristic function. Interestingly when both the genus and the size of the finite field grow we 
get the following central limit theorem.

\begin{thm}\label{thm1.4}
	(1). If $q$ is fixed and $g\rightarrow\infty,$ then : 
		\begin{equation*}
	\mathfrak{R}_{(\mathcal {L};g,q)}(H)-
	\log{\left(\frac{q^{(n^2-1)}}{\prod\limits_{k=2}\limits^{n}(q^{k-1}-1)(q^k-1)}\right)}
	+\delta_{\gamma/2}\sum\limits_{k=2}\limits^{n}\log(1- 1/q^{k})
	\end{equation*}
	converges weakly to a random variable $\mathcal{R},$ whose characteristic function 
	$\phi(\tau)=\mathbb{E}(e^{i\tau\mathcal{R}}) $ is given by 
	\begin{eqnarray*}
		\phi(\tau)&=& 1+ \sum\limits_{r=1}\limits^{\infty}\frac{1}{2^{r}r!}\sum\limits_{\substack{P_{j}\text{distinct}\\1\leq j\leq r}}\prod\limits_{j=1}\limits^{r}\left(\dfrac{(1-\mid P_{j} \mid^{-2})^{-i\tau}+ (1+\mid P_{j} \mid^{-2})^{-i\tau}-2}{(1+ \mid P_{j}\mid ^{-1})}\right),  \, \, 
	\end{eqnarray*}
	for all real number $\tau.$
	Here $P_{j}$'s are monic, irreducible polynomials in $\F_{q}[x]$ and $\mid P_{j}\mid= q^{\deg P_{j}}.$ $\delta_{\gamma/2}=1$ if $\gamma$ is even and $0$ otherwise.
	
	(2). If both $q, g\rightarrow\infty,$ then 
	the random variable
	${q^{3/2}\mathfrak{R}_{(g,q)}}$ is distributed as a standard Gaussian.
	More precisely, for each $z\in \mathbb{R},$ we have, 
	\begin{equation*}
	\lim\limits_{\substack{q\rightarrow\infty\\g\rightarrow\infty}} \mathbb{P} \left( {q^{3/2}\mathfrak{R}_{(g,q)}} \leq z\right) =\frac{1}{\sqrt{2\pi}}\int_{-\infty}^{z}e^{-\tau^{2}/2}d\tau .
	\end{equation*}
	\end{thm}

\begin{rmk}\label{rmk1.2}
	If we put $n=2$ and $d=1,$ then Theorem \ref{thm1.3} and Theorem \ref{thm1.4} are precisely Theorem $1.1$ (more accurately Remark $1$) and Theorem $1.2$ in \cite{ASA} for the case $M_{L}(2,1).$
	\end{rmk}

\begin{rmk}\label{rmk on quotient}
	Let $T_{n}$ be the group of $n$-division points of the jacobian $J_{X}.$ Suppose the characteristic of the field $F_{q}$ is co-prime with $n.$ 
	Then over $\F_{q},$ the group scheme $T_{n}$ acts on the moduli space $M_{L}(n,d)$ by tensorisation.
	Now $T_{n}$ being finite, the quotient $M_{L}(n,d)/T_{n}$ exists, and is a connected component of projective $PGl(n)$-bundles and is a projective scheme. Further from \cite[Theorem 2]{HaNa}, it follows that
	\[
	\w=N_{q}\bigg( M_{L}(n,d)/T_{n} \bigg).
	\] 
	Using the above identity, Theorem \ref{thm1.3} and Theorem \ref{thm1.4} can be proved mutatis-mutandis for the quotient $M_{L}(n,d)/T_{n}$.
	
\end{rmk}

Now, if we look at the case when rank and degree of the vector bundle are not coprime, the moduli space may not  be smooth. In particular, if we see the moduli space $M(2,0),$ this is smooth only when genus of the curve $X$ is two [cf. \cite{NaRa1}]. For genus $\geq 2$, Seshadri constructed a natural desingularisation (moduli theoretic) of $M(2,0)$ over any algebraically closed field $k,$ having characteristic other than $2$ (cf.\cite{Se}).
He constructed a smooth projective variety $N(4,0),$ whose closed points corresponds to the $S$-equivalence classes of parabolic stable vector bundles of quasi-parabolic type (4,3) together with small weights $(\alpha_1,\alpha_2)$ such that the underlying bundles are semi-stable of rank $4$ and degree $0$ semi-stable and having the endomorphism algebras as specialisations of $(2\times 2)$- matrix algebras.The natural desingularisation map from $N(4,0)$ to $M(2,0)$ is an isomorphism over $M^{s}(2,0)$ (see [Theorem 2, \cite{Se}] for more details). 
Furthermore, restricting on the subvariety of $N(4,0)$ whose closed points corresponds to the $S$-equivalence classes of parabolic stable vector bundles with the determinant of underlying bundle isomorphic to $\mathcal{O}_{X},$ is the desingularisation of $M_{\mathcal{O}_{X}}(2,0)$ (cf. Remark 1 in \cite{B1}, Theorem $2.1$ in \cite{BaSe1}),
and we denote this moduli space by ${\widetilde {N}}$.

In this paper, over the family of hyperelliptic curves $H_t$ in $\H_{\gamma,q}$, we study the distribution of the $\F_{q}$-rational points on $M^{s}_{\mathcal{O}_{H_{t}}}(2,0)$ and its Seshadri disingularisation model ${\N}.$ We have following similar results for these moduli spaces:

\begin{thm}\label{thm1.5}
	(1). If $q$ is fixed and $g\rightarrow\infty,$ then for $H_{t}\in \hh,$ the quantity 
	\[\log N_{q} (M^{s}_{\mathcal{O}_{H_{t}}}(2,0))  -3(g-1)\log{q}-\log{\left(\frac{q^3}{(q-1)^2(q+1)}\right)}+\delta_{\gamma/2}\log(1- 1/q^{2})\] 
	\noindent
	converges weakly to a random variable $\mathcal{R},$ whose characteristic function $\phi(\tau)=\mathbb{E}(e^{i\tau\mathcal{R}}) $ is given by 
	\begin{eqnarray*}
		\phi(\tau)&=&
		 1+ 
		 \sum\limits_{r=1}\limits^{\infty}
		 \frac{1}{2^{r}r!}\sum\limits_{\substack{P_{j}\text{distinct}\\1\leq j\leq r}}
		 \prod\limits_{j=1}\limits^{r}\left(\frac{(1-\mid P_{j} \mid^{-2})^{-i\tau}
		 	+ (1+\mid P_{j} \mid^{-2})^{-i\tau}-2}{(1+ \mid P_{j}\mid ^{-1})}\right),  
	\end{eqnarray*}
	for all real number $\tau.$
	
	(2). If both $q, g\rightarrow\infty,$ then for $H_{t}\in \hh,$ the random variable
	\[q^{3/2}\left(\log N_{q}( M^{s}_{\mathcal{O}_{H_{t}}}(2,0)) -3(g-1)\log{q}\right)\] is distributed as a standard Gaussian. More precisely, for any $z\in \mathbb{R},$ we have, 
	\begin{equation*}
	\begin{split}
	\lim\limits_{\substack{q\rightarrow\infty\\g\rightarrow\infty}}\frac{1}{\#\hh}\#\left\lbrace H_{t} \in \hh:q^{3/2}\left(\log N_{q}( M^{s}_{\mathcal{O}_{H_{t}}}(2,0) )  -3(g-1)\log{q}\right)\leq z\right\rbrace \\
	= \frac{1}{\sqrt{2\pi}}\int_{-\infty}^{z}e^{-\tau^{2}/2}d\tau.
	\end{split}
	\end{equation*} 
\end{thm}

\begin{rmk}\label{rmk1.3}
	For the moduli space $M_{\mathcal{O}_{H_{t}}}(2,0),$ 
	when $q$ is fixed and $g\rightarrow\infty,$ then over the family $\hh,$ one can see that the random variable
	\[\log N_{q} (M_{\mathcal{O}_{H_{t}}}(2,0))  -3(g-1)\log{q}-\log{\left(\frac{q^3}{(q-1)^2(q+1)}\right)}+\delta_{\gamma/2}\log(1- 1/q^{2})\] 
	converges weakly to a random variable $\mathcal{R},$ whose characteristic function is the same $\phi(\tau)=\mathbb{E}(e^{i\tau\mathcal{R}}) $ as given in Theorem \ref{thm1.5}. Furthermore, when both $q, g\rightarrow\infty,$ then the random variable
	\[q^{3/2}\left(\log N_{q}( M_{\mathcal{O}_{H_{t}}}(2,0)) -3(g-1)\log{q}\right)\] is distributed as a standard Gaussian. 
	\end{rmk}

Finally for the smooth projective variety ${\widetilde{N}}$, over the family of hyperelliptic curves we get,

\begin{thm}\label{thm1.6}
	(1). If $q$ is fixed and $g\rightarrow\infty,$ then for {$H_t\in \hhh,$} the random variable
	\[\log{N_{q}({\widetilde{N}})} -(4g-4)\log{q} +\delta_{\gamma/2}\log(1- 1/q^{2})\]
	converges weakly to a random variable $\mathcal{R},$ whose characteristic function $\phi(\tau)=\mathbb{E}(e^{i\tau\mathcal{R}}) $ is given by 
	\begin{eqnarray*}
		\phi(\tau)&=& 
		1
		+
		 \sum\limits_{r=1}\limits^{\infty}
		 \frac{1}{2^{r}r!}
		 \sum\limits_{\substack{P_{j}\text{distinct}\\1\leq j\leq r}}
		 \prod\limits_{j=1}\limits^{r}
		 \left(
		 \frac{(1-\mid P_{j} \mid^{-1})^{-i\tau}
		 	+
		 	 (1+\mid P_{j} \mid^{-1})^{-i\tau}-2}{(1+ \mid P_{j}\mid ^{-1})}
	 	 \right),  \, 
	\end{eqnarray*}
for all $\tau$ in $\mathbb{R},$ where $P_{j}$'s are monic, irreducible polynomials in $\F_{q^2}[x].$

	(2). If both $q, g\rightarrow\infty,$ then for $H_{t}\in \hhh,$ 
	\[q\left(\log{N_{q}({\widetilde{N}})} -(4g-4)\log{q}\right)\]
	\noindent
	is distributed as a standard Gaussian. That is for any $z\in \mathbb{R},$ we have, 
	\begin{align*}
	&\lim\limits_{\substack{q\rightarrow\infty\\g\rightarrow\infty}}
	\frac{1}
	{\# \hhh}
	\#\left\lbrace 
	 {H_t \in \hhh}:
	q\left(\log{N_{q}({\widetilde{N}})} -(4g-4) )\log{q}\right)
	\leq z
	\right\rbrace
	= \frac{1}{\sqrt{2\pi}}\int_{-\infty}^{z}e^{-\tau^{2}/2}d\tau.
	\end{align*} 
\end{thm}	
	
	We remark here that, all the results on distribution presented in this article  are over the family of hyperelliptic curves $(N=2)$. 
	So, it will be interesting to see whether the proof of Theorem \ref{thm1.4},\ref{thm1.5} and \ref{thm1.6} discussed in section \ref{distribution M(n,d)}, \ref{distribution Ms} and \ref{distribution N tilde} can be generalized to a more general set up, that is to study the distribution over a family of non-hyperelliptic curves $(N\geq 3)$.

The layout of the paper is as follows: In the first part of section \ref{preli}, we recall the definitions of zeta function over curves and the Artin $L$-series over function fields. In the second part we recall some basic properties of vector bundles, parabolic vector bundles and moduli space. We briefly describle the $\F_{q}$-rational points over these moduli spaces. In Section \ref{distribution M(n,d)}, using induction on the rank of vector bundles, first we give a bound on the number of isomorohism classes of unstable vector bundles (see Proposition \ref{Pdom5.1}), and using this we prove Theorems \ref{thm1.3} and \ref{thm1.4}. The proofs of Theorem \ref{thm1.5} and Theorem \ref{thm1.6} are in Section \ref{distribution Ms} and \ref{distribution N tilde} respectively.\\

\noindent
\textbf{Notation:} The notation $f(y)=O(g(y)),$ or equivalently, $f(y)\ll g(y)$ for a non-negative function $g(y)$ implies that there is a constant $c$ such that $|f(y)|\leq cg(y)$ as $y\rightarrow \infty.$ The notation $f(y)=o(g(y))$ is used to denote that $\frac{f(y)}{g(y)} \rightarrow 0$ as $y \rightarrow \infty.$ We use the notation $\mathbb{G}_{m}$ to denote the multiplicative group. 
We use the notation $J_{X}^{d}$ to denote the isomorphism classes of line bundles of  degree $d$ on a curve $X$.

\section{Preliminaries}\label{preli}
In this section we quickly recall some basic definitions and record some results which will be used later.

\subsection{Zeta functions of curves}   
\label{zeta function basics}
Let $\F_{q}$ be a finite field with $q$ elements and $\overline{\F}_{q}$ be its algebraic closure. 
Let $X$ be a smooth projective geometrically irreducible curve of genus $g \ge 1$ over $\F_{q}$ and $ \bar{X}\,=\,X \times_{\F_{q}} \overline{\F}_{q}$.

Given any integer $r \,>\,0$, let $\F_{q^r}\,\subset \, \overline{\F}_{q}$ be the unique field extension 
of degree $r$ over $\F_{q}$. Let $N_r\,:=\, N_{q^{r}}(\bar{X})$ be the cardinality of the set of $\F_{q^r}$- 
rational points 
of $X$. Recall that the zeta function of $X$ is defined by 
\begin{equation}
Z_{X}(t)\,=\,\exp\left(\sum_{r >0} \frac{N_rt^r}{r}\right).
\end{equation}
By the Weil conjectures it follows that the zeta function has the form 
\begin{equation}\label{zeta}
Z_{X}(t)\,=\,\dfrac{\prod\limits_{l=1}\limits^{2g}\left(1-\sqrt{q}e(\theta_{l,X})t\right)}{(1-t)(1-qt)},
\end{equation}
where $e(\theta)=e^{2\pi i\theta}.$

Further assume that $X$ is a Galois cover of $\mathbb{P}^1$ with Galois group $G=Gal(R/K)$ of order $N$, where $R:=\F_q(X)$ is the function field of $X$ over $\F_q$ and $K:=\F_q(x)$, the rational function field. For a prime
$\mathfrak{P}$ of $R$, the norm denoted by $|\mathfrak{P}|$ is the cardinality of the residue field of $R$ at $\mathfrak{P}$. The zeta
function $\zeta_R(s)$ is defined by 
\begin{equation}\label{defzeta1.1}
\zeta_R(s)=\prod\limits_{\mathfrak{P} \in R} \left( 1-\frac{1}{|\mathfrak{P}|^s}\right)^{-1}.
\end{equation}
We know that, for the rational function field $K$, prime ideals are in a one-one correspondence with the prime ideals in the polynomial ring $\F_{q}[x],$ only with one exception, and that is the prime at infinity say $P_{\infty}.$ Here $P_{\infty}$ is the discrete valuation ring generated by $\frac{1}{x}$ in $\F_{q}[\frac{1}{x}]$ such that $\deg(P_{\infty})=1.$ By above definition in equation \eqref{defzeta1.1}, the zeta function $\zeta_K(s)$ becomes
\begin{equation*}
\zeta_{K}(s)=\left( 1-\frac{1}{|{P_{\infty}}|^s}\right)^{-1}\prod\limits_{{P}\in \F_{q}[x]} \left( 1-\frac{1}{|{P}|^s}\right)^{-1}=\left( 1-\frac{1}{|{P_{\infty}}|^s}\right)^{-1} \zeta_{\F_{q}[x]}(s).
\end{equation*}
Now $|{P_{\infty}}|=q^{\deg(P_{\infty})}=q.$ And 
\begin{equation*}
\zeta_{\F_{q}[x]}(s)=\sum\limits_{\substack{f \,\,\text{monic}\\ \text{in} \,\, \F_{q}[x]}}\frac{1}{\mid f \mid ^{s}}=\left( 1-q^{1-s}\right)^{-1}. 
\end{equation*}
Therefore,
\[ \zeta_{K}(s)=\left(1-q^{-s}\right)^{-1}\left(1-q^{1-s}\right)^{-1}.
\]
Since $X$ is a smooth projective curve, the zeta function of the curve coincides with
the zeta function of it's function field (see \cite{Ro} for details). More precisely,
\[
Z_X(q^{-s})=\zeta_R(s).
\]
Henceforth we would use $\zeta_R$ and $\zeta_X$ interchangeably to 
denote this zeta function. 
From \eqref{zeta}, we get
\begin{eqnarray}\label{zeta2}
\zeta_{X}(s)=
\frac{\prod\limits_{l=1}\limits^{2g}
	\left(1-\sqrt{q}e(\theta_{l,X})q^{-s}\right)}{(1-q^{-s})(1-q^{1-s})}
=\zeta_K(s)\prod\limits_{l=1}\limits^{2g}\left(1-\sqrt{q}e(\theta_{l,X})q^{-s}\right). 
\end{eqnarray}
Next we recall the Artin L-series for function fields (cf. \cite[Chapter 9]{Ro} for more details). For each prime $P$ of $K$ and a prime $\fP$ of $R$ lying above $P$, 
we denote the Inertia group and the Frobenious element by
$I(\fP/P)$ and $(\fP, R/K)$ respectively.

Let $\rho$ be a representation of $G= \Gal (R/K)$
\begin{center}
	$\rho\,:G\rightarrow\,Aut(V)$
\end{center}
where $V$ is a vector space of dimension $n$ over complex numbers.
Let $\chi$ denotes the character corresponding to $\rho$. For an unramified prime $P$ and $\Re(s)>1$ we define 
the local factor by 
\begin{center}
	$L_P(s,\chi,K)=\det(I -\rho((\fP, R/K))\mid P \mid^{-s})^{-1}$.
\end{center}
Let $\{\alpha_{1}(P), \alpha_{2}(P),....\alpha_{n}(P)\}$ be the eigenvalues of 
$\rho((\fP, R/K)).$ 
In terms of these eigenvalues, we can rewrite the above expression as
\begin{equation}
L_P(s,\chi,K)\,=\prod\limits_{i=1}\limits^{n}
(1-\alpha_{i}(P)\mid P \mid^{-s})^{-1}.
\end{equation}
We note that these eigenvalues $\alpha_{i}(P)$ are all roots of unity 
because $(\fP, R/K)$ is of finite order.

For a ramified prime $P$, the local factor is defined as 
\[
L_P(s,\chi,K)=det(I-\rho((\fP, R/K))_H\mid P\mid^{-s})^{-1}
\]
where $\rho((\fP, R/K))_H$ denote the action of Frobenious automorphism restricted to a subspace $H$ of
$V$ fixed by $I(\fP/P)$.

In either case, we can write 
\begin{equation}\label{local_factor}
L_P(s,\chi,K)=\prod\limits_{i=1}\limits^{n}(1-\alpha_{i}(P)\mid P \mid^{-s})^{-1}.
\end{equation}
where each $\alpha_{i}(P)$ is either roots of unity or zero. 
The Artin $L$-series $L(s,\chi,K)$ or simply $L(s,\chi)$ is defined by
\begin{equation}\label{Artin-factors}
L(s,\chi)= \prod\limits_{P}L_P(s,\chi,K)
\end{equation} 
It is known that if $\rho=\rho_{0}$, 
the trivial representation, then $L(s,\chi)=\zeta_{K}(s),$ and if $\rho=\rho_{reg}$, 
the regular representation, then $L(s,\chi)=\zeta_{R}(s).$

Finally let $\{\chi{_{1}}, \chi_{2}, ... ,\chi_{h}\}$ be the set of irreducible characters 
of the Galois group $G$ with $\chi_{1}= \,\chi_{0},$ the trivial character.
For $i=1,\cdots h$, let $T_{i}=\chi_i(1)$  
be the dimension of the representation space corresponding to 
$\chi_{i}$. 
Then using properties of characters and Artin $L-$series, we get
\begin{equation}\label{l-function-product}
\zeta_{R}(s)\,=\,\zeta_{K}(s)\prod\limits_{i=2}\limits^{h}L(s,\chi_{i})^{T_{i}}  .
\end{equation}

\subsection{Vector bundles and parabolic vector bundles} 
\label{vector bundle basics}
Let $X$ be a smooth projective geometrically irreducible curve of genus $g \ge 2$ over $\F_{q}$ and set $ \bar{X}\,=\,X \times_{\F_{q}} \bar{\F}_{q}$. A vector bundle $E$ on $X$ is a locally free sheaf of $\mathcal O_{X}$-modules of finite rank, where $\mathcal O_{X}$ is the structure sheaf. If $F$ is a subsheaf of a locally free sheaf $E$ for which the quotient $E/F$ is torsion free (and so locally free since $X$ is a curve), then $F$ is called a vector subbundle of $E.$ Let $\bar{E}\, =\,E \times_{\F_{q}}\bar{\F}_{q}$ be the extension of $E$ to $\bar{X}$ over $\bar{\F}_{q}$.	The rank of the sheaf is denoted by $\text{rank}(E)$. A rank one locally free sheaf is called an invertible sheaf or a line bundle. The degree $\text{deg}(E)$ of a rank $n$ vector bundle $E$ is the degree of it's $n$-th exterior power line bundle 
$\bigwedge^n(E)$ which is also known as determinant line bundle of $E$. For any non-zero vector bundle $E$, the slope of a vector bundle $\mu(E)$ is the rational number $\frac{\text{deg}(E)}{\text{rank}(E)}$. The vector bundle $\bar{E}$ is called \emph{stable} (resp. \emph{semistable}) if for all proper subbundles $\bar{F}( \neq\, 0, \, \bar{E})$ we have $\mu(\bar{F})\,< \,\mu(\bar{E}) $ (resp. $\mu({\bar{F}})\,\leq \,\mu(\bar{E})$), otherwise it is called \emph{nonsemistable} or \emph{unstable} vector bundle. Over the field $\F_{q},$ a vector bundle $E$ is called stable (resp. semistable) if the corresponding extended vector bundle $\bar{E}$ over $\bar{\F}_{q}$ is stable (resp. semistable). 

Though the vector bundle $E$ is defined over non-algebraically closed field $\mathbb F_q$, $E$ contains a uniquely determined flag of sub-bundles defined over $\mathbb F_q$ 
\begin{center}
	$0\,=\,F_{0}\,\subsetneqq \,F_{1}\,\subsetneqq...\subsetneqq\, F_{m} \,=\,E$
\end{center}
satisfying following numerical criterion (see [Definition 1.3.10 ,\cite{HaNa}]). 
\begin{enumerate}
	\item $(F_{i}/F_{i-1})$'s are semistable for $i=1,...,m$.
	\item $\mu(F_{i}/F_{i-1})>\mu(F_{i+1}/F_{i})$ for $i=1,...,m-1.$
\end{enumerate}  
This filtration is often called as Harder-Narasimhan (H-N) filtration or canonical  filtration  of $E$. The length of the unique flag corresponding to $E$ is called \textit{the length} of $E$ and we denote it by $l(E).$ Also, we use the notation $E^{\vee}$ to denote the dual of the vector bundle $E.$ Assume $X$ is defined over any algebraically closed field (for our purpose $\overline{\mathbb F}_{q}$). For any rational number $\mu,$ let $C(\mu)$ denote the Artinian category of semistable vector bundles on $X$ of slope $\mu.$ For any object $E$ in $C(\mu),$ there is a strictly increasing sequence of vector subbundles  
\begin{center}
	$0\,=\,F_{0}\,\subsetneqq \,F_{1}\,\subsetneqq...\subsetneqq\, F_{m} \,=\,E$
\end{center}
satisfying
\begin{itemize}
	\item $(F_{i}/F_{i-1})$'s are stable for $i=1,...,m$.
	\item $\mu(F_{1})=\mu(F_{2}/F_{1})=...=\mu(F_{m}/F_{m-1})=\mu(E).$ 
\end{itemize} 

Such a series is called a \emph{Jordan-H$\ddot{\text{o}}$lder filtration} of $E.$ The integer $m$ is called the length of the filtration and the direct sum $\bigoplus \limits_{i} F_{i}/F_{i-1}$ is called the associated grading and is denoted by $gr E.$ Moreover, all the Jordan-H$\ddot{\text{o}}$lder filtration of $E$ have the same length and the associated grading $grE$ is uniquely determined upto isomorphism. Unlike H-N filration the Jordan-H$\ddot{\text{o}}$lder fiiltration is only for semistable bundles defined over algebraically closed field.  Moreover, two semistable bundles $E_{1}$ and $E_{2}$ are called \emph{$S-$equivalent} if $gr E_{1} \cong gr E_{2}$ (cf. \cite{Le}, \cite{Se2}). 
We denote the $S$- equivalence class of a vector bundle $E$ by $[E].$

Next we quickly recall the definition of parabolic vector bundles over a smooth projective curve defined over an algebraically closed field  followed by \cite{MeSe} and \cite{Se}.  Let $V$ be a vector bundle on $X$ and $t$ be a closed point on $X.$ A parabolic vector bundle $(V,\Delta,\alpha_{*})$ with parabolic structure at $t$ is a vector bundle $V$, a quasi--parabolic structure 
$\Delta\,=\,(l_1,\cdots, l_r)$ of length $r$ together with parabolic weights $\alpha_*\,=\,(\alpha_{1},\alpha_{2}, \cdots, \alpha_{r})$ 
is given by following data: 
\begin{itemize}
	\item a flag on the fibre $V_{t}$ denoted by $\Delta$ given by linear subspaces $F^{i}V_{t}$ of $V_{t},$
	\begin{equation*}
	V_{t}= F^{1}V_{t}\supsetneq F^{2}V_{t}\supsetneq ...\supsetneq F^{r}V_{t}
	\end{equation*} 
	such that $\text{rank} (F^{i}V_{t})=l_{i}$ and $l_{i} > l_{i+1}$ for all $i=1,2,...,r$ and
	\item an $r$-tuple $\alpha_*\,=\,(\alpha_{1}, \alpha_{2}, .., \alpha_{r})$ called the \emph{weights} of the parabolic structure attached to \\
	\noindent
	$F^{1}V_{t}, F^{2}V_{t},...,F^{r}V_{t}$ respectively, such that $0\leq\alpha_{1}<\alpha_{2}...<\alpha_{r}<1.$

\end{itemize}

The \emph{parabolic degree} denoted by $\textit{par-deg}\, V$ of $V$ is defined by
\begin{equation*}
\textit{par-deg} V = \deg V + \sum\limits_{i} (l_{i}-l_{i+1})\alpha_{i},
\end{equation*}
and the parabolic slope denoted by $\textit{par-}\mu(V)$ is the ratio $\frac{\textit{par-deg} V}{\text{rank} V}.$

Now, suppose $W_*\,=\,(W, \Delta',\alpha'_{*})$ and $V_*\,=\,(V,\Delta,\alpha_{*})$ are two parabolic vector bundles on $\bar{X}$ with parabolic structures at $t \in \bar{X}.$ Suppose $\Delta'$ and $\Delta$ are of length $m$ and $n$ respectively.  We say that $W_*$ is a parabolic subbundle of $V_*$ if following holds: 
\begin{enumerate}
\item{} $W$ is a subbundle of $V$ in the usual sense, and $\alpha'_*$ is a proper subset of $\alpha_*$ 
\item{} Moreover,  
for given $1\leq i_{0}\leq m,$ consider the greatest $1\le \, j_{0}\,\le n$ such that $F^{i_{0}}W_t \subset F^{j_{0}}V_t$ and $F^{i_{0}}W_t \,\nsubseteq \, F^{j_{0}+1}V_t$, then $\alpha'_{i_0}\,=\, \alpha_{j_0}$. 
\end{enumerate} 
A parabolic vector bundle $V$ is called parabolic stable (resp. semistable) if for all proper parabolic subbundles $W( \neq\, 0, \, V)$ of $V,$ we have $\text{par}\,\mu(W)\,< \,\text{par}\,\mu(V) $ (resp. $\leq $).
 
 For our purpose, we will be considering quai-parabolic structure $\Delta$ of the type $\left( 4,3\right) $ with small weights $(\alpha_{1}, \alpha _2)$ at a point $t \in X$ on a rank $4$ vector bundle $V$. 
So, the parabolic degree of $(V,\Delta,\alpha_*)$ is given by $\textit{par-deg} V=\deg V+\alpha_1+3\alpha_2$.
Let $W$ be a subbundle of $V$ of rank $r$, then $W$ acquires a canonical structure of a parabolic subbundle of $V$ as follows:\\
\noindent
{\bf Case I:} If $W_t\not\subset F^2V_t$. In this case we set, $$W_t=F^1W_t\supset F^2W_t \text{ and } F^2V_t\cap W_t=F^2W_t,$$ where weight of $F^1W_t=\alpha_1$ and $F^2W_t=\alpha_2$. In this case, $\text{rank}(F^2W_t)=\text{rank}(W)-1$ and \\ $\textit{par-deg} W= \deg W+\alpha_1+(r-1)\alpha_2$.\\
\noindent
{\bf Case II:} If $W_t\subset F^2V_t$. In this case we set, $W_t=F^1W_t\supset F^2W_t \text{ and weight of}\,F^1W_t=\alpha_2.$. Hence,  $\textit{par-deg} W=\deg W+r\alpha_2$. \\

\subsection{Moduli space of vector bundles}

Now we assume $X$ is defined over $\mathbb F_{q}$ and $L$ be a line bundle on $X$ of degree $d$  defined over $\mathbb F_{q}$.   Let $M(n,d)$ (resp. $M^{s}(n,d)$) be the moduli space of S-equivalence classes of semistable (resp. stable) vector bundles of rank $n$ and
degree $d$ over on $X$. It is well known that $M(n,d)$ is a normal, irreducible projective variety of dimension $n^{2}
(g-1)+1.$ Further, if $(n,d) \,=\,1,$ then definition of stability and semistability coincides and  $M(n,d)$ is smooth.   
By going to a finite extension of $\mathbb F_q$ if required,  we can assume all our varieties are defined over $\mathbb F_q$.   
There is a natural surjection given by the determinant map, $$\text{det}:\, M(n,d) \rightarrow J_{X}^{d}$$ which sends any bundle $E \in M(n,d)$ to it's determinant bundle $\Lambda^{n}(E) $. The inverse image $\text{det}^{-1}(L)$ is denoted by $M_{L}(n,d)$ and is called the \emph{fixed determinant moduli space} of semistable bundles of rank $n$ and determinant $L$. For the case $(n,d) \,=\, 1,\, M_{L}(n,d)$ is an irreducible smooth projective variety of dimension $(n^{2}-1)(g-1).$ For detailed structure of these varieties one can see \cite{DeRa}, \cite{DeRa1}.


\subsection{$\F_{q}-$Rational points of moduli spaces}
\label{$M(2,0)$ basics}
In this paper we will be interested in counting $\F_{q}-$rational points of following three moduli spaces: 
\begin{enumerate}
\item{} Moduli space $M_L(n,d)$ of rank $n$ and degree $d$ vector bundles with fixed determinant $L$, when $n$ and $d$ are coprime.
\item{} Moduli space $M^s_{\mathcal O_X}(2,0)$ of rank $2$ and degree $0$ stable vector bundles with fixed determinant $\mathcal O_X$. 
\item{} Sehsadri desingularization $\N$ of $M_{\mathcal O_X}(2,0)$. 
\end{enumerate}

To compute $\F_{q}$-rational points of these moduli spaces, following theorem which is known as {\it Siegel's formula} (see section $2.3$ in \cite{HaNa}, Proposition $1.1$ in \cite{DeRa} ) will play an important role: 

\begin{thm*}[Siegel's formula]	\begin{equation}\label{Si2.1}
	\sum\limits_{E\in\cM} \dfrac{1}{N_{q}(\Aut E)}\,=\, \dfrac{1}{q-1}q^{(n^2-1)(g-1)}\zeta_{X}(2)\zeta_{X}(3)...\zeta_{X}(n).
	\end{equation}
where $\mathcal{M}_{L}(n,d)$ denotes the set of all isomorphism classes of rank $n$ vector bundles on $X$ defined over $\mathbb F_q$  with the fixed determinant $L$, and $\Aut E$ denotes the group scheme of automorphisms of $E$ defined over $\F_{q}$.  
\end{thm*}

\vspace{.5 cm}

\noindent
{\bf Case 1: $\F_{q}-$Rational points of $M_{L}(n,d)$}, when $(n,d)\,=\,1$: When $E$ is a stable bundle defined over $X$, we know that $\Aut(E) \simeq \mathbb{G}_{m}$ (cf. \cite{Ne}) over $\F_{q}$, where $\mathbb{G}_{m}$ denote the multiplicative group, and hence the Siegel's formula in \eqref{Si2.1} asserts that, 
\begin{equation}\label{Sigen2.2}
N_{q}(\M)\,=\, q^{(n^2-1)(g-1)}\zeta_{X}(2)\zeta_{X}(3)...\zeta_{X}(n)\,-\,\sum\limits_{E\in\Mus} \dfrac{q-1}{N_{q}(\Aut(E))}.                        
\end{equation} 
where $\Mus\,=\,\cM \setminus \M.$ We know that, for any unstable vector bundle $E$ in $\Mus,$ it admits a unique Harder-Narasimhan filtration by subbundles (again defined over $\mathbb F_q$) 
\[ 0\, =\,E_{0}\, \subsetneqq E_{1}\,\subsetneqq...\,\subsetneqq E_{m}\,=\,E.\]
\noindent
We denote the numbers as,
$d_{i}(E):=\deg(E_{i}/E_{i-1}),\, r_{i}(E):=rank(E_{i}/E_{i-1})$ and $\mu_{i}(E):=\mu(E_{i}/E_{i-1}).$
Let
\[HN(n_{1},n_{2},...,n_{m}):=\left\lbrace E \in \Mus\,\vert \,l(E)=m, \, \text{and}\, r_{i}(E)=n_{i}\, \text{for}\, i=1, 2, ...,m\right\rbrace ,\]
\noindent
and 
\begin{equation}\label{unschain}
C_{L}(n_{1},n_{2},...,n_{m})\,:=\, \sum\limits_{E \in HN(n_{1},n_{2},...,n_{m})}\A.
\end{equation}
Then we see that,
\begin{equation}\label{Palpha}
\sum\limits_{E\in\Mus}\frac{1}{N_{q}(\Aut E)}\,=\, \sum\limits_{\substack{(n_{1},n_{2},...,n_{m})\\\sum\limits_{i=1}^{m}n_{i}=n,\, m\geq 2}} C_{L}(n_{1},n_{2},...,n_{m}).
\end{equation}
\noindent
Suppose $\beta_{L}(n,d)\,:=\,\sum\A,$ where the summation extends over isomorphism classes of semistable vector bundles $E$ on $X$ defined over $\mathbb F_q$ of rank $n$ with determinant $L$. 
We recall Proposition $1.7$ in \cite{DeRa} here, which is as follows.
\begin{prop}\label{Pbeta}
	\begin{itemize}
		\item[(i)]$\beta_{L}(n,d)$ is independent of $L$ and hence may be written simply as $\beta(n,d).$
		\item[(ii)]\begin{equation}
		C_{L}(n_{1},n_{2},...,n_{m})\,=\, \sum\frac{(N_{q}(J_{X}))^{m-1}}{q^{^\chi \(\tiny{
				\begin{matrix}
				n_{1} & n_{2} &...& n_{m}\\
				d_{1} & d_{2} &...& d_{m}  
				\end{matrix}
			}\)}}\prod\limits_{i=1}^{m}\beta(n_{i},d_{i})
		\end{equation}
		\noindent
		where the summation extends over $(d_{1},d_{2},...,d_{m})\in \mathbb{Z}^{m}$ with $\sum\limits_{i=1}^{m}d_{i}=d$ and $\frac{d_{1}}{n_{1}}>\frac{d_{2}}{n_{2}}>...>\frac{d_{m}}{n_{m}}.$\\
		\noindent
		Here,
		\begin{equation*}
		\chi \(\tiny{
			\begin{matrix}
			n_{1} & n_{2} &...& n_{m}\\
			d_{1} & d_{2} &...& d_{m}  
			\end{matrix}
		}\)=\sum\limits_{i<j}(d_{i}n_{j}-d_{j}n_{i})+\sum\limits_{i<j}n_{i}n_{j}(1-g).
		\end{equation*}
	\end{itemize}
\end{prop}

We end this section followed by the study of rank $3$ case. Based on the above proposition, 
we can compute 
\[
\sum\limits_{E\in \mathcal{M}_{\text{L}}^{\text{us}}(3,d)} \dfrac{1}{N_{q}(\Aut(E))}.
\]
This will be used as the first step in the induction hypothesis to prove Proposition \ref{Pdom5.1}. 
First note that, for any unstable vector bundle $E$ in $\mathcal{M}_{\text{L}}^{\text{us}}(3,d),$ it could be either in $HN(1,1,1),$ or $HN(1,2)$ or in $HN(2,1).$

\begin{prop}\label{unstable rank 3}
	 With the notation as above,
	 \begin{equation}\label{uns3}
	 C_{L}(1,1,1)= \frac{q^5\left( N_{q}(J_{X})\right)^{2} q^{3(g-1)}}{(q-1)^{3}(q^2-1)(q^3-1)},
	 \end{equation}
and
\begin{equation}\label{uns4.1}
C_{L}(2,1)=C_{L}(1,2)= \frac{q^{6}N_{q}(J_{X})q^{2(g-1)}}{(q-1)(q^6-1)}\left\lbrace \frac{2q^{3(g-1)}\zeta_{X}(2)}{(q-1)}-\frac{q^{g-1}N_{q}(J_{X})}{(q-1)^{3}(q+1)}-\frac{q^{g}N_{q}(J_{X})}{(q-1)^{3}(q+1)}\right\rbrace .
\end{equation}
\end{prop}

\begin{proof}
Using Proposition \ref{Pbeta}, we see that 
\begin{equation}\label{uns1}
C_{L}(1,1,1)\,=\, \sum\frac{(N_{q}(J_{X}))^{2}}{q^{\chi \(\tiny{
			\begin{matrix}
			1 & 1 & 1\\
			d_{1} & d_{2} &d_{3}
			\end{matrix}
		}\)}}\prod\limits_{i=1}^{3}\beta(1,d_{i}),
\end{equation}
where the summation extends over $(d_{1},d_{2},d_{3})\in \mathbb{Z}^{3}$ with $\sum\limits_{i=1}^{3}d_{i}=d$ and $d_{1}> d_{2} >d_{3}.$ Also,
\begin{equation*}
\chi \(\tiny{
	\begin{matrix}
	1 & 1 & 1\\
	d_{1} & d_{2} & d_{3}  
	\end{matrix}
}\)=2(d_{1}-d_{3}) + 3(1-g).
\end{equation*}
Since $\beta(1, d_{i})=\frac{1}{q-1}$ for $i=1,2,3,$ from \eqref{uns1}, we have 
\begin{equation}\label{uns2}
C_{L}(1,1,1)= 
\frac{
	\left( 
	N_{q}(J_{X})
	\right)^{2} 
	q^{3(g-1)}}{(q-1)^{3}}
\sum
\limits_{      d_{1}>d_{2}>d_{3}        }
\frac{1}{q^{2(d_{1}-d_{3})}
}.
\end{equation}
\noindent
Putting $d_{3}=d-d_{1}-d_{2}$, we see that,
\begin{eqnarray*}
	\sum\limits_{d_{1}>d_{2}>d_{3}}\frac{1}{q^{2(d_{1}-d_{3})}} &=& q^{2d}\sum\limits_{d_{1}>d_{2}>d-d_{1}-d_{2}}\frac{1}{q^{4d_{1}+2d_{2}}}\\
	&=& q^{2d}\sum\limits_{d_{1}>d_{2}}\frac{1}{q^{4d_{1}}}\sum\limits_{d_{2}>\frac{d-d_{1}}{2}}\frac{1}{q^{2d_{2}}}\\
	&=& \frac{q^{d}}{\left( 1-1/q^2\right) }\sum\limits_{d_{1}>\frac{d}{3}}\frac{1}{q^{3d_{1}}}\\
	&=& \frac{q^5}{(q^2-1)(q^3-1)}.
\end{eqnarray*}
Using this in \eqref{uns2}, we get \eqref{uns3}. 

Now, let $E$ be in $HN(2,1).$ Therefore, 
we have the Harder-Narasimhan filtration $ 0 \subsetneqq E_{1}  \subsetneqq E $, where $ r(E_{1})=2, r(E/E_{1})=1,$ and the subbundles $E_{1},$ and $E/E_{1}$ are semistable. Assume that $\deg(E_{1})= d_{1},$ and $\det(E_{1})=L_{1}.$ Therefore, $\deg(E/E_{1})=d-d_{1}$ and 
$\det(E/E_{1})=L \otimes L_{1}^{-1}.$ 
The Euler characteristic of $(E/E_{1}^\vee \otimes E_{1}),$ that is $ \chi(E/E_{1}^\vee \otimes E)\,=\, \chi \(\tiny{
	\begin{matrix}
	2 & 1\\
	d_{1} &d-d_{1}
	\end{matrix}
}\)=3d_{1}-2d + 2(1-g)$. 
Now, using Proposition \ref{Pbeta} we get,
\begin{eqnarray*}
	C_{L}(2,1)&=&\sum\limits_{\frac{d_{1}}{2}>d-d_{1}}\frac{N_{q}(J_{X})q^{2(g-1)+2d}}{q^{3d_{1}}}\beta(2,d_{1})\beta(1,d-d_{1})\\
	&=& \frac{N_{q}(J_{X})q^{2(g-1)+2d}}{q-1}\sum\limits_{d_{1}>\frac{2d}{3}}\frac{\beta(2,d_{1})}{q^{3d_{1}}}\\
	&=& \frac{N_{q}(J_{X})q^{2(g-1)+2d}}{q-1}\left\lbrace\beta(2,0)\sum\limits_{k>\frac{d}{3}}\frac{1}{q^{6k}} + \beta(2,1)\sum\limits_{k>\frac{2d-3}{6}}\frac{1}{q^{3(2k+1)}} \right\rbrace. \\
\end{eqnarray*} 
Therefore, 
\begin{equation}\label{uns4}
C_{L}(2,1)=
\frac{N_{q}(J_{X})q^{2(g-1)}}     {(q-1)\left(1-1/q^6 \right)}
\left\lbrace
\beta(2,0)
+ 
\beta(2,1)
\right\rbrace.
\end{equation} 
\noindent
Using \eqref{beta(2,0)} in \eqref{calbeta}, we see
\begin{equation*}
\beta(2,0)= \frac{q^{3(g-1)}}{(q-1)}\zeta_{X}(2)- \frac{ N_{q}(J_{X})q^{g-1}}{(q-1)^{3}(q+1)}.	
\end{equation*}
\noindent
Similarly, we compute $\beta(2,1)$ (cf. Proposition $2.2$ in \cite{ASA} for detailed computation) and putting the values of $\beta(2,0)$ and $\beta(2,1)$ in equation \eqref{uns4} we finally obtain \eqref{uns4.1}.

In a similar approach we compute $C_{L}(1,2)$ and one can see that it is equal to the quantity $C_{L}(2,1)$ . 
\end{proof}

\vspace{.5 cm} {\bf Case 2: $\mathbb F_q$-rational points of $M_{\mathcal{O}_{X}}^{s}(2,0)$}:

\noindent

Let $\mathcal{M}_{\mathcal{O}_{X}}(2,0) \left( \text{resp.}\,\mathcal{M}^{\text ss}_{\mathcal{O}_{X}}(2,0), \mathcal{M}^{\text us}_{\mathcal{O}_{X}}(2,0)\right) $ be the set of all isomorphism classes of rank $2$ vector bundles (resp.  semistable, unstable vector bundles) defined over $\mathbb F_q$ 
on $X$ with trivial determinant $\mathcal{O}_{X}$. 
For brevity we will denote them by $\mathcal M$ (resp. $\mathcal{M}^{\text {ss}}$, $\mathcal{M}^{\text {us}}$).  

As before, we denote the sum 
\begin{equation}\label{def:beta(2,0)}
\beta(2,0)\,:=\, \sum\limits_{E\in \mathcal{M}^{\text {ss}} }\frac{1}{N_{q}(\Aut(E))}
\end{equation}
and
\begin{equation*}\label{def:beta'(2,0)}
\beta'(2,0)\,:=\, \sum\limits_{E\in \mathcal{M}^{\text {us}}}\dfrac{1}{N_{q}(\Aut(E))}.
\end{equation*}
From the Siegel formula \eqref{Si2.1}, we have
\begin{equation}\label{calbeta}
\beta(2,0)+ \beta'(2,0)= \dfrac{q^{3g-3}}{q-1}\zeta_{X}(2).
\end{equation}
Now, for any $E$ in $\mathcal{M}^{\text us}_{\mathcal{O}_{X}}(2,0),$ it has the Harder-Narasimhan filtration of the type $HN(1,1)$. We have, 

\begin{equation}\label{beta(2,0)}		
	\beta'(2,0)=\frac{ N_{q}(J_{X})q^{g-1}}{(q-1)^{3}(q+1)} \,\,\,\, (\text{pg 7, eqn 4, \cite{BaSe1}}).
\end{equation}
\noindent

Since the stable bundles over $\mathbb F_q$ admit only scaler automorphisms, from \eqref{def:beta(2,0)}, we can write: 
\begin{equation}\label{betass}
\beta(2,0)\,=\, \frac{\u}{(q-1)} \,+\, \sum\limits_{E\in \mathcal{M}^{\text{ss}} \setminus \mathcal{M}^{\text{s}}}\frac{1}{N_{q}(\Aut (E))}.
\end{equation}
Now we try to compute the second term of the right hand side of \eqref{betass}. 
First we recall for convenience the definition of Kummer variety $K$ associated to the Jacobian $J_{\bar X}$, where $\bar X \,=\, X \times \bar{\mathbb {F}}_q$. In our setting  $K$ can be defined as follows: 

\[K:= \left\lbrace \text{Isomorphism classes of vector bundles of the form}\,\, \xi\oplus \xi^{-1} \right\rbrace ,\]
where $\xi$ is a degree zero line bundle over $\bar{X}$.  We have a canonical morphism $\phi:\, J_{\bar X} \, \rightarrow \, K$ defined by $ \xi \mapsto \xi \oplus \xi^{-1}$. 
It is known that $K$ has $2^g$ nodal singularities which we denote by $K_0$ and it is given by 
\begin{equation}\label{defK_{0}}
K_{0} :=\left\lbrace  \xi \oplus \xi: \xi^{2} \cong\mathcal{O}_{\bar{X}},\, \xi\in J_{\bar{X}}
\right\rbrace.
\end{equation}
Without loss of generalities we can assume all $\xi$ and the extensions are defined over $\mathbb F_q$. 
We also know that $M_{\mathcal O_{\bar X}}(2,0) \,- \,M^{s}_{\mathcal O_{\bar X}}(2,0) \cong K$ (cf. \cite{B1}). Moreover, for $E$ being a strictly semistable  vector bundle on $X$ defined over $\mathbb F_q$ with trivial determinant we have an exact sequence over $\bar{\mathbb F_{q}}$
\begin{equation}\label{extn1}
0\longrightarrow \xi\longrightarrow \bar{E}\longrightarrow \xi^{-1}\longrightarrow 0,
\end{equation}
 where $\bar{E}\,=\,E \times_{X} \bar{X}$ and the line bundle $\xi \in J_{\bar X}$ is uniquely determined. 
 
 From the above discussion, we get a surjective set theoretic map $\theta : \mathcal M^{ss}_{\mathcal O_{X}}(2,0)\setminus\mathcal M^{s}_{\mathcal O_{X}}(2,0) \rightarrow K$ which maps a semistable vector bundle $E$  \eqref{extn1} (defined over $\mathbb {F}_q$)  to $gr(\bar E)\,=\,\xi \oplus \xi^{-1}$. Apriori the object $gr(E)$ is defined over $\bar{\mathbb F}_q$ but since $E$ is defined over $\mathbb F_q$, one can show that (see [\S 3, \cite{BaSe1}]) both $\xi$ and $\xi^{-1}$ are defined either over $\mathbb F_q$ or $\mathbb F_{q^2}$. 
 
 Further one can show that (\S 3, \cite{BaSe1}),  
 \[
 \theta(\mathcal M^{ss}_{\mathcal O_{X}}(2,0) \,- \,\mathcal M^{s}_{\mathcal O_{X}}(2,0)) \,\, =\,\,A  \, \sqcup \, B \,\sqcup\, K_{0}.
 \]


Where $A$ and $B$ are defined as follows:
\begin{equation}\label{defA}
A:= \left\lbrace  \xi \oplus \xi^{-1} \in K\setminus K_{0} : \xi \,\text{and} \,\xi^{-1}\,\text{ are both defined over} \,\mathbb{F}_{q} \right\rbrace \,
\text{and}
\end{equation}
\begin{equation}\label{defB}
B:= \left\lbrace  \xi \oplus \xi^{-1}\in K\setminus K_{0}: \xi \,\text{and} \,\xi^{-1}\,\text{ are defined over $\mathbb {F}_{q^2} $ but not defined over} \,\mathbb{F}_{q} \right\rbrace.
\end{equation}
Clearly,

\begin{equation}\label{AB1}
{\#}(A)\,=\, \frac{1}{2}(N_q(J_{\bar{X}}) -2^{2g}), 
\end{equation}
and
\begin{equation}\label{AB2}
{\#}(B)\,=\,\frac{1}{2}(N_{q^2}(J_{\bar{X}}-J_0) - N_{q}(J_{\bar{X}}-J_0))
\,=\, \frac{1}{2}(N_{q^2}(J_{\bar{X}})-N_{q}(J_{\bar{X}})).
\end{equation}
\noindent


 We set,
\begin{equation}
\label{definition beta1}
\beta_{1}:=\sum\limits_{\substack{E\in \theta^{-1}(A) }}\frac{1}{N_{q}(\Aut(E)}\,+\, \sum\limits_{\substack{E\in \theta^{-1}(B) }}\frac{1}{N_{q}(\Aut(E))},
\end{equation}
and 
\begin{equation}
\label{definition beta2}
\beta_{2}:=\, \sum\limits_{\substack{E\in \theta^{-1}(K_0) }}\frac{1}{N_{q}(\Aut(E))}
\end{equation}
We have, 
\begin{equation}\label{def2:beta(2,0)}
\beta(2,0)\,=\, \frac{\u}{(q-1)}+ \beta_{1}+\beta_{2}.
\end{equation}
	Therefore, from \eqref{def2:beta(2,0)} and \eqref{calbeta}, we obtain
\begin{equation}\label{Ms1}
\u= q^{3g-3}\zeta_{X}(2)-\left\lbrace \beta'(2,0)+\beta_{1}+\beta_{2}\right\rbrace (q-1).
\end{equation}
We already know $\beta'(2,0)$. In order to know the left hand side we only need to know what is $\beta_1$ and 
$\beta_2$. Both the term $\beta_1$ and $\beta_2$ are computed in \cite{BaSe1}. While going through the proof given there we noticed in computation of $\beta_1$, they may have missed some term in consideration. Here, we first compute $\beta_{1}.$ We follow the same method as given in [Proposition 3.6, \cite{BaSe1}]. 
  
  Suppose $E\,\in\,\theta^{-1}(A)$, then there are two possibilities depending on the fact that the extension \eqref{extn1} being split or non-split. In the non-split case $\text{Aut}(E)\,=\, \mathbb G_m$  (over $\mathbb F_q$) and in the split case $\text{Aut}(E)\,=\, \mathbb G_m\,\times\, \mathbb G_m$  (over $\mathbb F_q$) (see [Lemma 3.3, \cite{BaSe1}). Also note that any  extension of $\xi$ with $\xi^{-1}$ 
 is semistable of degree $0$ and two such extension are isomorphic if and only if they are scaler multiple of each other in the extension space $H^1 (X,\xi^{-2})$ which is of dimension $g-1$.  Further if $\xi$ and $\xi'$ are not isomorphic then any two vector bundles $E \in H^1 (X,\xi^{-2})$ and $E' \in  H^1 (X,\xi'^{-2})$ are also non-isomorphic. 
All these facts together tells us that, 
  
  \[
  \sum\limits_{E\,\in\,\theta^{-1}(A)}\frac{1}{N_{q}(\Aut(E))}\,=\, \frac{\#(A)}{(q-1)^2} \,+\, \frac{2\#(A) N_q(\mathbb P^{g-2})}{(q-1)}.
  \]

  Now suppose $E \in \theta^{-1}(B)$. Then we know $ E \times _{\mathbb F_q} {F_{q^2}} \,=\, \xi \oplus \xi^{-1}$, where $\xi \in J_{\bar{X}} -J_0$ defined over $\mathbb F_{q^2}$. In this case $\text{Aut}(E) \,=\, \mathbb G_m $ over $\mathbb F_{q^2}$ [Lemma 3.5,\cite{BaSe1}]. Clearly if $E_1$ and $E_2$ are in $\theta^{-1}(B)$ and are not isomorphic over $\mathbb F_q$  then they split not ismorphically over $\mathbb F_{q^2}$. Hence, 
  
  \[
  \sum\limits_{E\,\in\,\theta^{-1}(B)}\frac{1}{N_{q}(\Aut(E))}\,=\, \frac{\#(B)}{(q^2-1)}
  \]
  
  Hence we have,
  \begin{equation}\label{Beta1}
  \beta_1\,=\,  \frac{\#(A)}{(q-1)^2} \,+\, \frac{2\#(A) N_q(\mathbb P^{g-2})}{(q-1)} \,+\, \frac{\#(B)}{(q^2-1)}.
  \end{equation}
  
  Now, from Proposition 3.1 in \cite{BaSe1}, we have
 \begin{equation}\label{Beta2}
 	\beta_{2}\,=\,\frac{2^{2g}}{N_{q}(GL(2))}+ \frac{2^{2g}N_{q}(\mathbb{P}^{g-1})}{q(q-1)}.
 	\end{equation}

Putting together the results from \eqref{beta(2,0)}, \eqref{Beta1}, and \eqref{Beta2} along with the size of $A$ and $B$ as in \eqref{AB1} and \eqref{AB2} in \eqref{Ms1} we obtain the following estimate of $\u.$
\begin{prop}\label{Ms2.1}
	The number of $\F_{q}$-rational points of the moduli space of stable bundles $\u$ is given by the following expression:
	\begin{align*}\label{Ms}
	\u= q^{3g-3}\zeta_{X}(2)-\frac{\left( q^{g+1}-q^2+q\right) }{(q-1)^2(q+1)}N_{q}(J_{X})-\frac{1}{2(q+1)}N_{q^2}(J_{X})+ \frac{1}{2(q+1)}2^{2g}.
	\end{align*}
\end{prop}

%

\noindent
{\bf Case3: $\mathbb F_q$-rational points of $\N$}:
First we give a brief description of the Seshadri desingularisation model $N(4,0)$ for the moduli space $M(2,0)$ following Seshadri (\cite{Se}), and
for that, we introduce some notations here.
\begin{itemize}
	\item $V_{2} \, (\text{resp} \,\, V_{2}^{s}):=$ The category of rank $2$ degree $0$ vector bundles ( resp. stable vector bundles) on $\bar{X}$.\\
	
	\item $PV_{4}  \, (\text{resp}\,\,PV_{4}^{ss},\, PV_{4}^{s}):=$ The category of parabolic  vector bundles $(V, \Delta, \alpha_*)$ (resp. parabolic semistable vector bundles, parabolic stable vector bundles)  with the parabolic structure $\Delta$ at a fixed point $p$ in $\bar{X}$ of the type $(4,3)$ and parabolic weights $(\alpha_{1}, \alpha_{2})$ such that the underlying vector bundle $V$ is of rank $4$ and degree $0.$\\
	
	\item $PU^{ss}(4,0)$ (resp. $PU^{s}(4,0)$):= The moduli space of parabolic semistable vector bundles which is a scheme of finite type of dimension $4g,$ whose closed points corresponds to the $S-$ equivalence classes of parabolic  semistable (resp. parabolic stable) vector bundles in the category $PV_{4}^{ss}$ (resp. $PV_{4}^{s}$). This scheme corepresent the parabolic functor defined in \cite{MeSe}.
\end{itemize}
\noindent

\noindent
We choose the weights $(\alpha_1,\alpha_2)$ to be sufficiently  small so that we have:
\begin{enumerate}
	\item The notion of parabolic semistability and parabolic stability coincide.
	\item Parabolic stable implies the underlying vector bundle is semistable.
	\item If $(V,\Delta, \alpha_*)\in PV_4^{s}$ then for every sub-bundle $W$ of $V$ with the induced parabolic structure, we have $\text{deg}(W)<0$ implies $\textit{par-deg } W<0$.
\end{enumerate}
%
%
%
%
\noindent
For any $W$ in $V_{2}^{s},$ we get a unique (upto isomorphism) parabolic structure $\Delta$ of the type $(4,3)$ with weights $(\alpha_{1}, \alpha_{2})$ sufficiently small such that $(W\oplus W, \Delta(W), \alpha_*)$ is in $PV_{4}^{s}$ (see Proposition $1$ in \cite{Se}). Also using Proposition $1$ in \cite{Se}, we see that the association 
\begin{align*}
&V_{2}^{s}: \longrightarrow PV^{s}_{4}\\
&W \longmapsto (W\oplus W, \Delta, \alpha_* )
\end{align*}
\noindent
is a well define set theoretic injective map say, $\zeta_{2}^{s}.$ Moreover, the map $\zeta_{2}^{s}$ is the underlying map on closed points of a morphism of finite type between $M^{s}(2,0)$ and $PU^{s}(4,0)$ which sends any element $[W]$ in $M^{s}(2,0)$ to $[(W\oplus W), \Delta, \alpha_*]$ in $PU^{s}(4,0)$ (cf. Proposition $6.6$, \cite{TEVB}), and we denote it again by $\zeta_{2}^{s}$.\\
\noindent
Let $PR^s$ be a parabolic reduced quote scheme. We see that $PU^s(4,0)$ is the geometric quotient under a free action of $PSL(n,\F_{q})$ for a certain $n$ on $PR^s$. Let $q: PR^{s} \rightarrow PU^{s}(4,0)$ be this geometric quotient morphism. This is an open map. 
\noindent
Let $QR^{s}$ be the reduced closed subscheme in $PR^{s}$  whose closed points corresponds to the subset
\[\{(V, \Delta)\in PR^s\mid \dim(\End(V))=4\}. \]
Now let the scheme-theoretic image $q(QR^s)$ in $PU^s(4,0)$ is denoted by $RU^s(4,0).$ It can be shown that $RU^s(4,0)$ is closed subscheme in $PU^s(4,0)$ and  $PR^s$ being  reduced, $RU^s(4,0)$ is also reduced.\\
\noindent
Now let $NR^s$ be a maximal open subscheme in $QR^s$ whose closed points corresponds to the parabolic vector bundles such that the underlying semi- stable vector bundles have the endomorphism algebras ishomorphic to the $2 \times 2$ matrix algebras. We denote the image $q(NR^{s})$ by $N^{s}(4,0),$ which is an open subsheme of $RU^{s}(4,0)$ and let $N(4,0)$ denote the closure of $N^s(4,0)$ in $RU^s(4,0).$ Here we state the following result by Seshadri \cite{Se}.
\begin{thm}[Seshadri]\label{Seshadrithm}
	There is a natural structure of a smooth projective variety on $N(4,0),$ and there exist a canonical morphism $\pi_{2}: N(4,0) \rightarrow M(2,0)$ which is an isomorphism over $M^{s}(2,0).$ More precisely, if $(V,\Delta, \alpha_*)$ is in $N(4,0),$ then $gr(V)=W \oplus W,$ with $rank(W)=2$ and $W$ is a direct sum of line bundles of degree $0,$ and the morphism $\pi_{2}$ sends $(V, \Delta, \alpha_*)$ to $W$ in $M(2,0).$ Further $(V, \Delta, \alpha_*)$ is in $\pi_{2}^{-1}(M^{s}(2,0))$ if and only if $\,V=W\oplus W,$ where $W$ is in $M^{s}(2,0)$ or equivalently $\End(V)$ is isomorphic to $2\times 2$ matrix algebra.  
\end{thm}

Now let ${\N}$ be the closed subscheme of $N(4,0)$ whose closed points corresponds to the subset 
\[
\left\lbrace (V, \Delta, \alpha_*) \in N(4,0)\mid \det(V)=\mathcal{{O}}_{X}\right\rbrace.
\]
One can check easily that ${\N}$ is the desingularisation of $M_{\mathcal{O}_{X}}(2,0)$ (cf. Theorem $2.1$ in \cite{BaSe1}), and we denote the desingularisation between them again by $\pi_{2}.$ 
Using results from \cite{BaSe1}, \cite{B1} and \cite{BaSe}, we have
\begin{equation}\label{pointN2.2}
N_{q}({\N}) \,=\, \u + N_{q}(Y) + 2^{2g}N_{q}(R)+ 2^{2g}N_{q}(S),
\end{equation}
where 
$R$ is a vector bundle of rank $(g-2)$ over $G(2,g),$ the Grassmanian of $2$ dimensional subspaces of $g$ dimensional vector space,
and $S$ is isomorphic to $G(3,g),$ and
$Y$ is a $\mathbb{P}^{g-2}\times \mathbb{P}^{g-2}$ bundle over $K\setminus K_{0},$ 
such that
\begin{equation}\label{pointY}
N_{q}(Y) \,=\, \#(A) N_{q}(\mathbb{P}^{g-2}\times \mathbb{P}^{g-2}) + \#(B)N_{q^{2}}(\mathbb{P}^{g-2}),
\end{equation}
(cf. Proposition 4.1 in \cite{BaSe}) where $A$ and $B$ are the set defined as in \eqref{defA} and \eqref{defB} respectively.

\section{Distribution on $M_{L}(n,d)$ } \label{distribution M(n,d)}

Let $X$ be a Galois curve of degree $N$ and genus $g \geq 2$ over a finite field $\F_{q}$ and $M_{L}(n,d)$ be the moduli space of stable vector bundles on $X$ of rank $n$ with fixed determinant $L$ of degree $d,$ such that $gcd(n,d)=1.$ 

\subsection{Proof of Theorem \ref{thm1.3}}
With the same set up as in subsection \ref{zeta function basics}, we recall that, for the smooth projective Galois curve $X,$ $R=\F_q(X)$ is a geometric Galois extension 
of $\F_q(x)$ with Galois
group $G$ of order $N$.  From equation (\ref{zeta2}), and (\ref{l-function-product}), we get
\[
\prod\limits_{i=2}\limits^{h}L(s,\chi_{i})^{T_{i}}
=\prod\limits_{l=1}\limits^{2g}(1-\sqrt{q}e(\theta_{l,X})q^{-s}).
\]
Using (\ref{local_factor}) and (\ref{Artin-factors}) on the left hand side, we get
\begin{equation}\label{local-Artin5.1}
\prod_{i=2}^h \prod\limits_{P}\prod\limits_{j=1}\limits^{T_{i}}(1-\alpha_{i,j}(P)\mid P\mid^{-s})^{-T_i}
=\prod\limits_{l=1}\limits^{2g}(1-\sqrt{q}e(\theta_{l,X})q^{-s})
\end{equation}
where the product is over all monic, irreducible polynomials $P\;\text{in}\; \F_{q}[x],$
and $P=P_\infty$, with $|P| = q^{\deg P}$ and $\deg(P_\infty) = 1,$ that is, $|P_\infty| = q.$ 
These $\alpha_{i,j}(P)$'s are either roots of unity or zero.

 Now, Taking logarithms on both sides of \eqref{local-Artin5.1} and equating the coefficient of $q^{-ms}$ 
for any positive integer $m$, we obtain 
\begin{equation}\label{theta-lambda-equality5.1}
-q^{m/2}\sum_{l=1}^{2g}e(m\theta_{l,X})=\sum\limits_{\deg\, f = m}
\Lambda(f)\sum\limits_{i=2}\limits^{h}T_{i}\sum\limits_{j=1}\limits^{T_{i}}\alpha_{i,j}(f)
\end{equation}
where the sum on the right is over all monic irreducible polynomials of degree $m$ over $\F_q$ and $\Lambda $ is the analogue of Von Mangoldt function defined as
\begin{equation*}
\Lambda(f):= \begin{cases}
\deg P &\text{ if } f=P^{k}\text{ for some monic, irreducible } P \in \F_{q}[x]\\  
0  &\text{  otherwise.}
\end{cases}
\end{equation*}
From the definition of zeta function as in \eqref{zeta2}, for any integer $k\geq2$ we can write 
\begin{eqnarray*}
	\zeta_{X}(k)&=& \frac{q^{(2k-1)}\prod\limits_{l=1}\limits^{2g}\(1-q^{-(2k-1)/2}e(\theta_{l,X})\)}{(q^{k}-1)(q^{k-1}-1)}.
\end{eqnarray*}
Taking logarithm on both sides of the above equation we get
\begin{equation}\label{zetadef5.2}
\log \zeta_{X}(k) - (2k-1)\log {q} + \log \left( (q^{k}-1)(q^{k-1}-1)\right)  \,=\, \sum\limits_{l=1}\limits^{2g}\log (1-q^{-(2k-1)/2}e(\theta_{l,X})). 
\end{equation}
For any positive integer $Z$, we define,
\begin{equation}\label{def: epsilon1}
 \epsilon_{1,Z} := -\sum\limits_{m\leq Z}q^{-(2k-1)m/2}m^{-1}\sum\limits_{l=1}\limits^{2g}e(m\theta_{l,X}),
 \end{equation}
and 
\begin{equation}\label{def: epsilon2}
\epsilon_{2,Z} := -\sum\limits_{m\geq Z+1}q^{-(2k-1)m/2}m^{-1}\sum\limits_{l=1}\limits^{2g}e(m\theta_{l,X}).
\end{equation}
Putting these in \eqref{zetadef5.2}, we can write for any integer $Z,$
\begin{equation}\label{log-zeta5.1}
\log{\zeta_{X}(k)} - (2k-1)\log{q} + \log\left( (q^k-1)(q^{k-1}-1)\right)  \,=\,\epsilon_{1,Z} + \epsilon_{2,Z}.
\end{equation}  
In the next result we estimate $\epsilon_{1,Z}$ and $\epsilon_{2,Z}$.
\begin{lem}\label{epsilon-Z5.1}
	For  $Z \geq 2$, we have
	\[| \epsilon_{1,Z}| \, \leq \,(N-1)\left( \frac{1}{q-1} + \frac{1}{q^{k}} (1.5 + \log{Z}-\log{2})\right).\] 
	\noindent
	and 
	\[| \epsilon_{2,Z}| \, \leq \, \frac{2g}{(Z + 1)} \, \frac{1}{q^{(2k-1)(Z + 1)/2}} \, \frac{1}{(1-q^{-(2k-1)/2})}.\]
	Moreover,
	\[| \epsilon_{1,1}| \, \leq \, (N-1)\left( \frac{1}{q^{k-1}} + \frac{1}{q^{k}}\right) \]
	and
	\[| \epsilon_{2,1}|\, \leq\, \frac{2g}{q^{(2k-1)}-q^{(2k-1)/2}}.\]
\end{lem}

\begin{proof}
	
	Let $Z\,\geq\, 2$. Using (\ref{theta-lambda-equality5.1}) in the definition of $\epsilon_{1, Z}$ in \eqref{def: epsilon1}, we get
	\begin{eqnarray*}
		| \epsilon_{1, Z}|
		&\leq & \sum\limits_{m\leq Z}q^{-k m}m^{-1}\sum\limits_{deg f = m}\Lambda(f)\sum\limits_{i=2}\limits^{h}T_{i}^{2}. 
	\end{eqnarray*}
	Now, using the property that $\sum\limits_{i=2}\limits^{h}T_{i}^{2}=N-1,$ and 
	$\sum\limits_{\deg\, f = m}\Lambda(f)=q^{m}+1,$ we have
	\begin{eqnarray*}
		| \epsilon_{1,Z}| &\leq &  \sum\limits_{m\leq Z}q^{-k m}m^{-1}(q^{m}+1)(N-1) \\
		&=& (N-1)\left( \sum\limits_{m\leq Z}\dfrac{1}{mq^{(k-1)m}} + \sum\limits_{m\leq Z}\dfrac{1}{mq^{km}}\right). 
	\end{eqnarray*}
Also,
\begin{equation*}
\sum\limits_{m\leq Z}\dfrac{1}{mq^{m}} \leq -\log(1-q^{-1})\leq \frac{1}{q-1},
\end{equation*}
which gives
\begin{eqnarray*}
	|\epsilon_{1,Z}| &\leq & (N-1)\left( \frac{1}{q^{(k-1)}-1} + \frac{1}{q^{k}}\sum\limits_{m\leq Z}\dfrac{1}{m}\right)  \\
	& \leq & (N-1)\left( \frac{1}{q^{k-1}-1} + \frac{1}{q^{k}}(1.5 + \log{Z} -\log{2})\right).
\end{eqnarray*}
Similarly, from the definition of $\epsilon_{2, Z}$ in \eqref{def: epsilon2}, we have
\begin{eqnarray*}
	| \epsilon_{2,Z}| &= & 
	\bigg| 
	\sum\limits_{m\geq Z + 1}q^{-(2k-1)m/2} m^{-1}\sum\limits_{l=1}\limits^{2g}-e(m\theta_{l,X})
	\bigg| \\
	&\leq & 2g\sum\limits_{m\geq Z+1}q^{-(2k-1)m/2}m^{-1} \\
	&\leq & \frac{2g}{(Z + 1)} \, \frac{1}{q^{(2k-1)(Z + 1)/2}} \, \frac{1}{(1-q^{-(2k-1)/2})}.
\end{eqnarray*}
Also, for $Z=1,$ trivially we get the desired bound for $\epsilon_{1,1},$ and $\epsilon_{2,1},$ and that completes the proof of Lemma \eqref{epsilon-Z5.1}.
\end{proof}

\begin{prop}\label{log-zeta5.2}
	With all the notations as above, for suitable absolute constants $c_{1}>0$ and $c_{2}>0$ and assuming $\log(g)>\kappa \log(Nq)$ for a sufficiently large absolute constant $\kappa >0$ (independent of $N$), we obtain
	\begin{equation*}
	\left|  \log{\zeta_{X}(k)}\right|  \leq \frac{c_{1}N}{\sqrt{q}}+\frac{c_{2}N\log{\log{g}}}{q^{k}}
	\end{equation*}
	for any $k\geq 2.$
\end{prop}
\begin{proof}
	If  $ \frac{2g}{(q^{2k-1}-q^{(2k-1)/2})}\, \geq \, q^{-1/2}(N-1)$, 
	using Lemma \ref{epsilon-Z5.1} with  
	\[Z =\frac{2}{3}\,\frac{\log{\left( \frac{2g\sqrt{q}}{(1-q^{-(2k-1)/2})(N-1)} \right) }}
	{\log{q}}\, \geq 2
	\]
	we get
	\begin{align*}
	&| \epsilon_{1,Z}| \leq (N-1)\left( \frac{1}{q-1} + \frac{1}{q^{k}}(1.5 -\log{2})\right)\\ 
	&+ \frac{(N-1)}{q^{k}}\log{\left(\frac{2}{3}\,\frac{\log{\left( \frac{2g\sqrt{q}}{(1-q^{-(2k-1)/2})(N-1)} \right) }}
		{\log{q}}\right)}.
	\end{align*}
	and
	\[
	| \epsilon_{2,Z}|  \leq \frac{(N-1)}{q^{k}}\log\left(\frac{2}{3}\frac{\log{\left( \frac{2g\sqrt{q}}{(1-q^{-(2k-1)/2})(N-1)} \right) }}
	{\log{q}} \right) 
	\]
Note that $\frac{1}{q-1}\leq \frac{1}{q}+\frac{2}{q^2}.$
Hence in this case
\begin{equation}\label{Z2}
| \epsilon_{1,Z}| + | \epsilon_{2,Z}|
\leq \,(N-1)
\Bigg(
\frac{1}{q} +\frac{2}{q^2}+ 
\frac{2}{q^{k}}
\log{
	\Bigg(
	\frac{2\log{
			\left( \frac{2g\sqrt{q}}{(N-1)(1-q^{-(2k-1)/2})}
			 \right) }}       {3\log{q}}
	\Bigg)
} 
+ \frac{1}{q^{k}}
\Bigg). 
\end{equation}
Now, suppose $ \frac{2g}{(q^{2k-1}-q^{(2k-1)/2})}\, < \, q^{-1/2}(N-1)$. 
Again using Lemma \ref{epsilon-Z5.1} with $Z=1,$ we get
\begin{equation}\label{Z1}
| \epsilon_{1,Z}|+ | \epsilon_{2,Z}| \leq 
(N-1)\left\lbrace \frac{1}{\sqrt{q}}+\frac{1}{q^k}+\frac{1}{q^{k-1}}\right\rbrace .
\end{equation}
Using \eqref{Z2}, and \eqref{Z1} in \eqref{log-zeta5.1}, we get 
\begin{align}\label{log-zeta2}
&\mid \log{\zeta_{X}(k)} -(2k-1)\log{q} + \log\{(q^k-1)(q^{k-1}-1)\}\mid \nonumber \\
&\le(N-1)\left( \frac{c}{\sqrt{q}} + 
\frac{2}{q^{k}}\left\lbrace 
\log{\left(\frac{2\log{\left( \frac{2g\sqrt{q}}{(N-1)(1-q^{-(2k-1)/2})} \right) }}{3\log{q}}\right)} \right\rbrace\right)
\end{align}
for some absolute constant $c> 0.$
After simplifying we can write
\begin{eqnarray*}
	\frac{2\log{\left( \frac{2g\sqrt{q}}{(N-1)(1-q^{-(2k-1)/2})} \right) }}{3\log{q}}
	&=& \frac{2\log{g}}{3\log{q}} \left(1 + O\left(\frac{\log{(Nq)}}{\log{g}} \right)  \right) .
\end{eqnarray*}

Hence,
\begin{align*}
\frac{(N-1)}{q^{k}}\log \left( \frac{2\log{g}}{3\log{q}} \left(1 + O\left(\frac{\log{(Nq)}}{\log{g}} \right)  \right) \right)
&= \frac{(N-1)\log{\log{g}}}{q^{k}} + O\left( \frac{N}{q^{k-1}}\right)  + O\left( \frac{N}{q^{k}}\frac{\log{(Nq)}}{\log{g}}\right) 
\end{align*}

under the assumption $ \log{g}> \kappa \log{(Nq)}$ for a sufficiently large constant $\kappa >0.$\\

Also, we see that, 
\begin{equation*}
\log(q^{k}-1)(q^{k-1}-1)
= (2k-1)\log{q}\,+\,O\left( \frac{1}{q^{k-1}}\right) .
\end{equation*}
Using the above two results in \eqref{log-zeta2}, we get,
\begin{equation*}
\left| \log\zeta_{X}(k)+ O\left(\frac{1}{q^{k-1}} \right) \right|\leq \frac{c(N-1)}{\sqrt{q}}+\frac{2(N-1)\log{\log{g}}}{q^{k}} + O\left( \frac{N}{q^{k-1}}\right)  + O\left( \frac{N}{q^{k}}\frac{\log{(Nq)}}{\log{g}}\right).
\end{equation*}

After simplification we can rewrite this as,
\[ \left|  \log{\zeta_{X}(k)}\right|  \leq \frac{c_{1}N}{\sqrt{q}}+\frac{c_{2}N\log{\log{g}}}{q^{k}}\]

with suitable choice of two constants $c_{1}>0$ and $c_{2}>0$ and hence the Proposition.

\end{proof}

From Proposition \ref{log-zeta5.2}, we obtain the following result. 
\begin{prop}\label{log-zeta5.3}
	There exist an absolute constant $c'>0$ such that
	\[
	(\log\,g)^{\frac{-c'N}{q^k}}\exp\left(\frac{-c'N}{\sqrt{q}} \right) \leq\zeta_{X}(k)\leq (\log\,g)^{\frac{c'N}{q^k}}\exp\left( \frac{c'N}{\sqrt{q}}\right), 
	\] 
	for any $k\geq 2,$ whenever $\log{g}>\kappa \log{(Nq)}$ for a sufficiently large constant $\kappa>0.$ 
\end{prop}


Now we recall the following result of Xiong and Zaharescu (Theorem 1 of \cite{XiZa}).
\begin{lem}\label{Xiong-Zaha}
	Let $X$ be a Galois curve of degree $N$ of genus $g \geq 1$ over $\F_{q}.$
	Then
	\begin{equation}\label{log-Jacobian}
	\left. \mid\log{\left( N_{q}(J_{X})\right)}-g\log{q}\right. \mid\, 
	\leq\, (N-1)\left(\log{\max\left\lbrace 1,\dfrac{\log{\left(\frac{7g}{(N-1)}\right)}}{\log{q}}\right\rbrace }
	+3\right).
	\end{equation}
\end{lem}
A direct consequence of Lemma\ref{Xiong-Zaha} is as follows.
\begin{prop}\label{log_Jac5.1}
	For $\log g>\kappa_{1} \log q,$ where $\kappa_{1}>0$ is a large constant, we have
	\[
	 q^{g}(\log\frac{7g}{N-1})^{-3(N-1)}\leq N_{q}(J_{X})\leq q^{g}(\log\frac{7g}{N-1})^{3(N-1)}.
	 \] 
\end{prop}

Next we recall the definition of $C_{L}(n_{1},n_{2},...,n_{k})$ in equation \eqref{unschain}. 
Based on the above mentioned two results on the bounds of the quantity $N_{q}(J_{X})$ and $\zeta_{X}(k)$,  we get the following result.
\begin{prop}\label{Pdom5.1}
	Let $E$ be an unstable vector bundle in $\Mus.$ Then for any partition $(n_{1},n_{2},...,n_{k})$ of the rank $n,$ we have
	$C_{L}(n_{1},n_{2},...,n_{k})=o\left(q^{(n^2-\frac{n}{c})(g-1)}\zeta_{X}(2)\zeta_{X}(3)...\zeta_{X}(n-1) \right)$ for all $k\geq 2,$ and $c>1.$
\end{prop}
\begin{proof}
	We use induction on the rank $n$ of the vector bundle.
For $n=3,$
	using Proposition\ref{log-zeta5.3} and Proposition \ref{log_Jac5.1} in \eqref{uns4.1}, we have
	\[
	\lim_ {g \rightarrow \infty}\dfrac{C_{L}(2,1)}{q^{(3^2-\frac{3}{c})(g-1)}\zeta_{X}(2)}=0.
	\]
	Similarly one can show that $C_{L}(1,2)$ and $C_{L}(1,1,1)$ are of the size $o(q^{(3^2-\frac{3}{c})(g-1)}\zeta_{X}(2)).$ Hence the induction hypothesis holds for the initial case.

	Now, suppose the statement is true for all partitions of $m,$ where $m$ is the rank of any vector bundle $E$ in $\mathcal{M}_{L}^{\text{us}}(m,d),$ and $m < n.$ That is for any $m<n,$ 
	\[C_{L}(m_{1},m_{2},...,m_{k})=o\left(q^{(m^2-\frac{m}{c})(g-1)}\zeta_{X}(2)\zeta_{X}(3)...\zeta_{X}(m-1) \right)\]
	\noindent
	for all $k\geq 2$ such that $\sum\limits_{i=1}^{k}m_{i}=m$. We want to show that the statement is true for rank $n.$
Suppose $E$ is in $HN(n_{1},n_{2},...,n_{m}),$ such that $E$ has the H-N filtration 
\[
 0=E_{0} \subsetneqq E_{1}...\subsetneqq E_{m}=E.
 \]
We write $M=E/E_{1}.$ Using Proposition $4.4$ in \cite{NaSe} it can be shown that, there is no nonzero homomorphism from $E_{1}$ to $M.$ Hence every automorphism of $E$ keeps $E_{1}$ invariant and it goes down to an automorphism of the quotient $M.$ Thus we get a well defined map :
\[
\Phi: \Aut E \rightarrow \Aut E_{1} \times \Aut M,
\]
such that, for any $f$ in $\Aut E,$ $\Phi(f)=(f\vert _{E_{1}}, p\circ f)$ where $\xymatrix{E\ar[r]^{f} & E\ar[r]^{p} & M}.$ Note that for any $g \in \Hom(M,E_{1}),$ we have the image $\Phi(Id+g)=(Id\vert_{E_{1}}, Id\vert_{M})$ for $\left( Id + g\right) $ in $\Aut E.$ Therefore, $ Id + H^{0}(\bar{X}, \Hom(M, E_{1}))$ is contained in $\ker(\Phi).$ Conversely let any $f \neq Id,$ is in $\ker(\Phi),$ that is\\ 
\noindent
$\Phi(f)=(Id\vert_{E_{1}}, Id\vert_{M}).$ We see that $(f-Id)\vert _{E_{1}}=0$ and $p\circ (f-Id)=0,$ and since $\Hom(E_{1}, M)=0,$ therefore, $(f-Id)$ is in $\Hom(M, E_{1}).$ So any $f$ in $\ker(\Phi),$ is in $Id +H^{0}(\bar{X}, \Hom(M, E_{1})), $ and hence $\ker(\Phi) \cong H^{0}(\bar{X}, \Hom(M, E_{1}))=T$(say).

Next we consider the action of the group $\Aut E_{1} \times \Aut M$ say $G$ on $ H^{1}(\bar{X}, \Hom(M, E_{1})),$ that is, on equivalence classes of extensions of $M$ by $E_{1}.$ For simplicity we denote $H^{1}(\bar{X}, \Hom(M, E_{1}))$ by $S.$ We denote any equivalene class of short exact extensions 
\[
\xymatrix{0\ar[r] & E_{1}\ar[r]^{\alpha} & E\ar[r]^{\beta} & M\ar[r] & 0}
\]
 in $S$ by $[E; (\alpha, \beta)].$ For any $(\phi_{1}, \phi_{2})$ in $G,$ define the action by 
\[
(\phi_{1}, \phi_{2})\cdot [E; (\alpha, \beta)]= [E; (\alpha\circ \phi_{1}, \phi_{2}\circ\beta)].
\]
We know that, any two extensions $(E; (\alpha_{1}, \beta_{1}))$ and $(E; (\alpha_{2}, \beta_{2}))$ are isomorphic if the following diagram commute:
\begin{center}

	\begin{tikzpicture}[descr/.style={fill=white,inner sep=2.5pt}]
\matrix (m) [matrix of math nodes, row sep=3em,
column sep=3em]
{ 0 & E_{1} & E & M & 0\\
	0& E_{1} & E & M & 0\\ };
\path[->,font=\scriptsize]
(m-1-1) edge node[auto] {} (m-1-2)
(m-1-2) edge node[auto] {$ \alpha_{1} $} (m-1-3)
(m-1-3) edge node[auto] {$ \beta_{1} $} (m-1-4)
(m-1-4) edge node[auto] {} (m-1-5)
(m-1-2)edge node[left]  {$ \phi_{1} $} (m-2-2)
(m-1-3)edge node[left] {$\phi $} (m-2-3)
(m-1-4)edge node[left] {$ \phi_{2} $} (m-2-4)
(m-2-1) edge node[auto] {} (m-2-2)
(m-2-2) edge node[auto] {$ \alpha_{2} $} (m-2-3)
(m-2-3) edge node[auto] {$ \beta_{2} $} (m-2-4)
(m-2-4) edge node[auto] {} (m-2-5);
\end{tikzpicture}
\end{center}
where $\phi_{1}, \phi$ and $\phi_{2}$ are isomorphisms. When $\phi_{1}=Id$ and $\phi_{2}=Id$, we get the equivalence class. From the definition, it is easy to see that, two such extensions in $S$ are isomorphic if and only if they are in the same orbit under this action and the isotropy subgroup of $[E; (\alpha, \beta)]$
denoted by $G_{[E; (\alpha, \beta)]}$is same as the image of $\Aut E$ under the map $\Phi.$ We denote the orbit of $[E; (\alpha, \beta)]$ in $S$ by $G\cdot [E; (\alpha, \beta)].$ Therefore, we have
\begin{equation*}
C_{L}(n_{1},n_{2},...,n_{m})= \sum\limits_{E_{1},M}\,\sum\limits_{[E; (\alpha, \beta)] \in S}\frac{1}{\vert \Aut E\vert \, \vert G \cdot [E; (\alpha, \beta)]\vert}.
\end{equation*}
Also, $\vert G \cdot [E; (\alpha, \beta)]\vert = \left[ G:G_{[E; (\alpha, \beta)]}\right] = \frac{\vert \Aut E_{1} \times \Aut M \vert }{\vert Img(\Phi)\vert}= \frac{\vert \Aut E_{1} \times \Aut M \vert \, \vert \ker(\Phi)\vert}{\vert \Aut E\vert}.$ 
So, 
\begin{equation*}
C_{L}(n_{1},n_{2},...,n_{m})= \sum\limits_{E_{1},M}\frac{\vert S \vert }{\vert \Aut E_{1}\vert \, \vert \Aut M\vert \, \vert T \vert} = \sum\limits_{E_{1},M}\frac{1}{\vert \Aut E_{1}\vert \, \vert \Aut M\vert q^{\chi(M^{\vee}\otimes E_{1})}},
\end{equation*}
where $\chi(M^{\vee}\otimes E_{1})$ denotes the Euler characteristic of $(M^{\vee}\otimes E_{1}).$ The summation extends over all pairs of bundles $(E_{1},M)$ where $E_{1}$ is semistable of rank $n_{1},$ and $M$ has H-N filtration of length $(m-1)$ and has determinant equal to $L \otimes (\det E_{1})^{-1}$ such that $\mu(E_{1})> \mu_{1}(M)>...>\mu_{m-1}(M)$ with $r_{i}(M)=n_{i+1},\, \, i=1, 2,...,m-1.$ 

Now, let $J_{X}^{d_{1}}(\F_{q})$ be the variety of isomorphism classes of line bundles of  degree $d_{1}$ on $X$ over $\F_{q}$ and $\det E_{1}=L_{1}$ of degree $d_{1}.$ Therefore,
\begin{equation}\label{induction1}
C_{L}(n_{1}, n_{2},...,n_{m})= \sum\limits_{\substack{d_{1} \in \Z\\ \frac{d_{1}}{n_{1}}> \frac{d_{2}}{n_{2}}}} \sum\limits_{L_{1} \in J_{X}^{d_{1}}(\F_{q})} \frac{C_{L \otimes L_{1}^{-1}}(n_{2}, n_{3},...,n_{m})}{q^{\chi \(\tiny{
			\begin{matrix}
			n_{1} &n- n_{1}\\
			d_{1} & d-d_{1}
			\end{matrix}
		}\)}} \sum\frac{1}{\mid \Aut E_{1} \mid},
\end{equation}
where the last sum on the right hand side of the above equation is over isomorphism classes of semistable bundles $E_{1}$ of rank $n_{1}$ with fixed determinant $L_{1}$ of degree $d_{1},$ which is nothing but $\beta_{L_{1}}(n_{1}, d_{1})$ as defined before.\\
Also,
 \[
 \chi\left(M^{\vee}\otimes E_{1} \right) =  \chi \(\tiny{
	\begin{matrix}
	n_{1} &n- n_{1}\\
	d_{1} & d-d_{1}
	\end{matrix}
}\)
=nd_{1}-n_{1}d-n_{1}(n-n_{1})(g-1).
\]
So, for any partition $(n_{1},n_{2},...,n_{k})$ of $n,$ using Proposition \ref{Pbeta}$(i)$ in \eqref{induction1}, we write 
\begin{align*}
	C_{L} (n_{1},n_{2},...,n_{k})
	&= \sum\limits_{\substack{d_{1}\in \Z \\ \frac{d_{1}}{n_{1}}>\frac{d}{n}}}\, \frac{N_{q}(J_{X})q^{n_{1}(n-n_{1})(g-1)+n_{1}d}}{q^{nd_{1}}}\beta(n_{1},d_{1})C_{L\otimes L_{1}^{-1}} (n_{2},...,n_{k}).
\end{align*}

Using the Siegel's formula \eqref{Si2.1} in the definition of $\beta(n_{1},d_{1}),$ we get 
\begin{align*}
\beta(n_{1},d_{1})
&=
 \frac{1}{q-1}q^{(n_{1}^2-1)(g-1)}\zeta_{X}(2)\zeta_{X}(3)...\zeta_{X}(n_{1})\\
&+
 \sum\limits_{\substack{(p_{1},p_{2},...,p_{l})\\\sum\limits_{i=1}^{l}p_{i}=n_{1},\, l\geq 2}} C_{L_{1}}(p_{1},p_{2},...,p_{l}).
\end{align*}
Using induction hypothesis for rank $n_{1}<n,$ we have
\begin{align}\nonumber
\beta(n_{1},d_{1})&= \frac{1}{q-1}q^{(n_{1}^2-1)(g-1)}\zeta_{X}(2)\zeta_{X}(3)...\zeta_{X}(n_{1})\\\label{beta-n5.1}
&+
 o\left(S(n_{1})q^{(n_{1}^2-\frac{n_{1}}{c})(g-1)}\zeta_{X}(2)\zeta_{X}(3)...\zeta_{X}(n_{1}-1) \right),
\end{align}
where $S(n_{1})$ is the number of partitions of $n_{1}.$
Also, 
\begin{equation*}
\sum\limits_{\substack{d_{1}\in \Z \\ \frac{d_{1}}{n_{1}}>\frac{d}{n}}}\,\frac{1}{q^{nd_{1}}}\,=\,
O\bigg( 
\frac{1}{q^{dn_{1}}}
\bigg). 
\end{equation*}

Using \eqref{beta-n5.1} we can write,
\begin{align*}
	C_{L} (n_{1},n_{2},...,n_{k})
	&= O\left( \frac{N_{q}(J_{X})}{(q-1)}q^{(n^2+n_{1}^2-nn_{1}-\frac{n}{c}+\frac{n_{1}}{c}-1)(g-1)}\mathcal{A} \right)\\
	&+ O\left( S(n_{1})N_{q}(J_{X})q^{(n^{2}+n_{1}^{2}-nn_{1}-\frac{n}{c})(g-1)}\frac{\mathcal{A}}{\zeta_{X}(n_{1})}\right).  
\end{align*}
where 
\[
\mathcal{A}:= \zeta_{X}(2)\zeta_{X}(3)...\zeta_{X}(n_{1})\zeta_{X}(2)\zeta_{X}(3)...\zeta_{X}(n-n_{1}-1).
\]
Using Propositions \ref{log-zeta5.3} and \ref{log_Jac5.1}, we get
\begin{align*}
&\frac{C_{L} (n_{1},n_{2},...,n_{k})}{q^{(n^2-\frac{n}{c})(g-1)}\zeta_{X}(2)\zeta_{X}(3)...\zeta_{X}(n-1)}\\
&= O\left(q^{-1} q^{(n_{1}^2+\frac{n_{1}}{c}-nn_{1})(g-1)}(\log g)^{N\left( \frac{2c'(n-n_{1}-1)}{q^2}+3\right)}  \exp\left(\frac{2c'N}{\sqrt{q}}  (n-n_{1}-1) \right) \right)\\
&+ O\left(q^{(n_{1}^2-nn_{1}+1)(g-1)}(\log g)^{N\left( \frac{2c'(n-n_{1}-1)}{q^2}+3\right)} \exp\left(\frac{2c'N}{\sqrt{q}} \left( n-n_{1}-1\right) \right) \right).
\end{align*}
This completes the proof as the right hand side tends to $0$ for $g$ tending to $\infty.$
\end{proof}

Furthermore, using \eqref{Palpha} and Proposition \ref{Pdom5.1} we can rewrite the equation \eqref{Sigen2.2} as follows, 

\begin{rmk}\label{Siegel3}
	\begin{equation*}
	\begin{split}
	N_{q}(M_{L}(n,d)) \,=\,q^{(n^2-1)(g-1)}\zeta_{X}(2)\zeta_{X}(3)...\zeta_{X}(n)\\
	\,+\, o\left((q-1)S(n)q^{(n^2-\frac{n}{c})(g-1)}\zeta_{X}(2)\zeta_{X}(3)...\zeta_{X}(n-1) \right).  
	\end{split}                      
	\end{equation*}
\end{rmk} 
\noindent
for any constant $c>1.$\\
\noindent
\textbf{Final step of the proof of Theorem \ref{thm1.3}:}\\
\noindent
From Remark \eqref {Siegel3}, 
\noindent
Let \[T_{1}:=q^{(n^2-1)(g-1)}\zeta_{X}(2)\zeta_{X}(3)...\zeta_{X}(n)\]
\noindent
and 
\[T_{2}:=o\left((q-1)S(n)q^{(n^2-\frac{n}{c})(g-1)}\zeta_{X}(2)\zeta_{X}(3)...\zeta_{X}(n-1) \right).\]
\noindent 
We choose the constant $c$ such as $1<c<n.$
Taking logarithm on both sides of $T_{1},$ we get  
\begin{equation*}
\log{T_{1}} = (n^2-1)(g-1)\log{q} + \sum\limits_{k=2}\limits^{n}\log\zeta_{X}(k).
\end{equation*}
Using Proposition \ref{log-zeta5.2} with some constant $c^{\prime}=\max\left\lbrace c_{1}, c_{2}\right\rbrace ,$ and Proposition \ref{log-zeta5.3}, we observe that
\begin{align*}
& |\log{\left(N_{q}(M_{L}(n,d))\right)}- (n^2-1)(g-1)\log q|\\
&\le c'(n-1)N\left( \frac{1}{\sqrt{q}}+ \frac{\log{\log{g}}}{q^2}\right) 
+ O\left((q-1)S(n)q^{(1-\frac{n}{c})(g-1)}\left(\log{g} \right)^{\frac{c'N}{q^{n}}}\exp\left( \frac{c'N}{\sqrt{q}}\right)\right). 
\end{align*}
Now, if we consider $A:=N\left( \frac{1}{\sqrt{q}}+ \frac{\log{\log{g}}}{q^2}\right) $
and
$B:=(\log g)^{\frac{N}{q^{2}}}\exp\left( \frac{N}{\sqrt{q}}\right),$ then $\log B=A.$
Therefore, 
\begin{equation*}
 \log{\left(N_{q}(M_{L}(n,d))\right)}- (n^2-1)(g-1)\log q \,=\, O_{n,N}(A+q^{-\sigma g}B),
 \end{equation*}
 for a suitable absolute constant $\sigma >0,$ depending on $n.$
The theorem follows up on simplifying.

\subsection{Proof of Theorem \ref{thm1.4}}  
Next we focus our attention on the family of hyperelliptic curves $\H_{\gamma,q}.$ 
Let $H_{t}$ be a hyperelliptic curve of genus $g\geq 2$ given by the affine model  $H_{t}: y^{2} = F_{t}(x)$ 
with $F_{t}$ in $\H_{\gamma,q}$. 
Suppose $(n, d)=1.$ Corresponding to a hyperelliptic curve $H_{t},$ we use the notation $M_{L_{t}}(n,d)$ for the moduli space of stable vector bundles of rank two with fixed determinant $L$ of degree $d$ defined over $H_{t}.$
The function field $\F_{q}(H_{t})$ corresponding to the hyperelliptic curve $H_{t}$ is a Galois extension of the rational function field $\F_{q}(x)$ of degree two.
We denote $\F_{q}(H_{t})$ by $K'$ and $\F_{q}(x)$ by $K$ for simplicity. Let $\chi=(\frac{.}{F_{t}})$ 
denote the Legendre symbol generating 
$Gal(K'/K).$
As discussed in subsection \ref{zeta function basics}, we have
\begin{equation}\label{L-chi}
L(s,\chi)=\prod\limits_{l=1}\limits^{2g}(1-\sqrt{q}e(\theta_{l,H_{t}})q^{-s}).
\end{equation}

The Euler product of $L-$function is given by
\begin{equation}\label{Euler}
L(s,\chi) = \prod\limits_{P}(1 - \chi{(P)}\mid P \mid ^{-s})^{-1}.
\end{equation}

Taking logarithmic derivatives of (\ref{L-chi}) and (\ref{Euler}) and equating coefficients of $q^{-ms}$
we get

\begin{eqnarray*}
	\sum\limits_{l=1}\limits^{2g}-e(m\theta_{l,H_{t}})
	&=& \sum\limits_{f \in \mathcal{F}_{m}}q^{-m/2}\Lambda(f)\left(\frac{F_{t}}{f}\right)
\end{eqnarray*}
where $\mathcal F_m$ is the set of all monic polynomials of degree $m$. 
Now for any $F_{t} \;\text{in}\; \hh,$ 
\begin{equation*}
\left(\frac{F_{t}}{P_\infty}\right) =
\begin{cases}
& 1  \text{ if } \deg(F_{t})\equiv 0 \,(\text{mod}\, 2)\\
& 0   \text{ otherwise}. 
\end{cases}
\end{equation*}
Defining $\delta_{\gamma/2}=1$ if $\gamma$ is even and $0$ otherwise, we note that
\[\sum\limits_{\substack{f =\infty \\ }}q^{-m/2}\Lambda(f)\left(\frac{F_{t}}{f}\right) 
=\, q^{-m/2}\delta_{\gamma/2}. 
\]

So
\begin{equation}\label{e1}
\sum\limits_{l=1}\limits^{2g}-e(m\theta_{l,H_{t}})
= \sum\limits_{\substack{f \neq \infty  \\ \deg f=m}}q^{-m/2}\Lambda(f)\left(\frac{F_{t}}{f}\right) 
+ q^{-m/2}\delta_{\gamma/2}.
\end{equation}
Proceeding as in the proof of Theorem \ref{thm1.3} we get 
\begin{equation}\label{R5.2}
\begin{split}
\log{\left(N_{q}\left( M_{L_{t}}(n,d)\right) \right)} -(n^2-1)(g-1)\log{q}
=\sum\limits_{k=2}^{n}\log{\zeta_{H_{t}}(k)} \\
+ O\left((q-1)S(n)q^{(1-\frac{n}{c})(g-1)}\left(\log{g} \right)^{\frac{c'N}{q^{n}}}\exp\left( \frac{c'N}{\sqrt{q}}\right)   \right). 
\end{split}
\end{equation}
for some constant $1<c<n.$

For a fixed positive integer $Z$, we write
\begin{equation}\label{total sum}
\sum\limits_{k=2}^{n}\log{\zeta_{H_{t}}(k)} - \sum\limits_{k=2}^{n}(2k-1)\log{q} + \sum\limits_{k=2}^{n}\log\{(q^k-1)(q^{k-1}-1)\}=\epsilon_{1,Z} + \epsilon_{2,Z}
\end{equation}
	where,
\begin{equation}\label{gen5.1}
\epsilon_{1,Z} = -\sum\limits_{m\leq Z}\left( \sum\limits_{k=2}^{n}q^{-\frac{(2k-1)}{2}m}\right) m^{-1}\sum\limits_{l=1}\limits^{2g}e(m\theta_{l,X}),
\end{equation}

and
\begin{equation}\label{gen5.2}
\epsilon_{2,Z} = -\sum\limits_{m> Z}\left( \sum\limits_{k=2}^{n}q^{-\frac{(2k-1)}{2}m}\right) m^{-1}\sum\limits_{l=1}\limits^{2g}e(m\theta_{l,X}).
\end{equation}

Observe that, \begin{equation}\label{R5.1}
\sum\limits_{k=2}^{n}q^{-\frac{(2k-1)}{2}m} =q^{-3m/2}(1+O(1/q)).
\end{equation}

Using \eqref{e1} and \eqref{R5.1} in \eqref{gen5.1}, we get
\begin{eqnarray*}
	\epsilon_{1,Z} &=&\sum\limits_{m\leq Z}q^{-3m/2} m^{-1}\sum\limits_{\substack{f \neq \infty \\ \deg f=m}}q^{-m/2}\Lambda(f)\left(\frac{F_{t}}{f}\right)\\ 
	&+& 
	\sum\limits_{m\leq Z}
	\bigg( 
	\sum\limits_{k=2}^{n}q^{-\frac{(2k-1)}{2}m}
	\bigg)
	q^{-m/2}m^{-1}\delta_{\gamma/2}\\
	&+&
	O
	\bigg(
	\frac{1}{q}\sum\limits_{m\leq Z}q^{-2m}m^{-1}\sum\limits_{\substack{f \neq \infty  \\ \deg f=m}}
	\Lambda(f)\left(\frac{F_{t}}{f}\right)
	\bigg).
	\end{eqnarray*}

Let \begin{equation}\label{main5.1}
\bigtriangleup_{Z}(F_{t}):= 
\sum\limits_{m\leq Z}q^{-2m}m^{-1}\sum\limits_{\substack{f \neq \infty  \\ \deg f=m}}
\Lambda(f)\left(\dfrac{F_{t}}{f}\right).
\end{equation}

We see,
\begin{equation}\label{upper-bounds5.1}
\sum\limits_{m\leq Z}q^{-2m}m^{-1}\sum\limits_{\substack{f \neq \infty  \\ \deg f=m}}
\Lambda(f)\left(\frac{F_{t}}{f}\right) \leq \frac{1}{q}(1+ \log{Z}).
\end{equation}
Also,
 \begin{equation*}
\sum\limits_{m\leq Z}
\bigg( 
\sum\limits_{k=2}^{n}q^{-\frac{(2k-1)}{2}m}
\bigg)
q^{-m/2}m^{-1}\delta_{\gamma/2}= -\delta_{\gamma/2}\sum\limits_{k=2}^{n}\log(1-1/q^k)-\delta_{\gamma/2}\sum\limits_{m>Z}\frac{1}{m}\sum\limits_{k=2}^{n}\frac{1}{q^{km}}.
\end{equation*}
After simplification we can write, 
\begin{equation*}
\epsilon_{1,Z}= \bigtriangleup_{Z}(F_{t}) 
+ 
O
\bigg( 
\frac{1}{q^2}
\left( 
1 + \log{Z}
\right) 
\bigg)  
-\delta_{\gamma/2}
\sum\limits_{k=2}^{n}\log(1-1/q^k)
-\delta_{\gamma/2}
\sum\limits_{m>Z}\frac{1}{m}
\sum\limits_{k=2}^{n}\frac{1}{q^{km}}.
\end{equation*}
Therefore rearranging the terms in \eqref{total sum} and based on above estimates for $\epsilon_{1, Z}$, we obtain
\begin{align}
&\sum\limits_{k=2}^{n}\log{\zeta_{H_{t}}(k)} \nonumber
- \sum\limits_{k=2}^{n}(2k-1)\log{q} 
+ \sum\limits_{k=2}^{n}\log\{(q^k-1)(q^{k-1}-1)\} 
+ \delta_{\gamma/2}\log\left\lbrace\prod\limits_{k=2}^{n} (1-1/q^k)\right\rbrace\\
&=\bigtriangleup_{Z}(F_{t}) + \epsilon_{Z,F_{t}}+O\left( 1/q^2\left( 1+\log{Z}\right) \right) ,
\label{gen5.3}
\end{align}

where, 
\begin{equation*}
\epsilon_{Z,F_{t}}=\epsilon_{2,Z} - \delta_{\gamma/2}\sum\limits_{m>Z}\frac{1}{m}\sum\limits_{k=2}^{n}\frac{1}{q^{km}}.
\end{equation*}
Using \eqref{gen5.3}, we write equation \eqref{R5.2} in a simpler form 
\begin{equation}\label{N_F5.1}
N_{F_{t}} = \bigtriangleup_{Z}(F_{t}) + \epsilon_{Z}(F_{t})
\end{equation}
where \begin{equation*}\label{total5.1}
\begin{split}
N_{F_{t}}:= \log{\left(N_{q}\left( M_{L_{t}}(n,d)\right) \right)} -(n^2-1)(g-1)\log{q} 
- \log{\left(\frac{q^{(n^2-1)}}{\prod\limits_{k=2}^{n}(q^{k-1}-1)(q^k-1)}\right)} \\+ \delta_{\gamma/2}\log{\left(\prod\limits_{k=2}^{n}1-1/q^k \right)},
\end{split}
\end{equation*}

and
\begin{equation}\label{error5.1}
\epsilon_{Z}(F_{t}):= \epsilon_{Z,F_{t}} + O\left(q^{(1-\frac{n}{c})(g-1)}\left(\log{g} \right)^{\frac{c'N}{q^{n}}}\exp\left( \frac{c'N}{\sqrt{q}}\right)\right) +O(1/q^2(1+\log{Z})).
\end{equation}
It is easy to see that, 
\begin{equation}\label{upper-bounds5.2}
\mid\epsilon_{Z,F_{t}}\mid = \bigcirc\left(\frac{g}{Z} q^ {-3Z/2}\right).
\end{equation}
{\bf The $r$th moment of $\bigtriangleup_{Z}$:}\\
For any function $\phi: \H_{\gamma, q} \rightarrow \C$, we denote the mean value of $\phi$ by 
\[
\left\langle  \phi  \right\rangle := \frac{1}{\#\H_{\gamma, q}} \sum\limits_{F_{t}\in \H_{\gamma, q}} \phi(F_{t}).
\]
If we see the proof of Theorem $1.2$ in \cite{ASA}, the quantity $\bigtriangleup_{Z}(F_{t})$ in equation \eqref{main5.1}, is exactly same as in equation $(4.5)$ in \cite{ASA}. Using similar procedure as in \cite{ASA}, for a positive integer $r \leq \log \gamma,$ and the choice of 
$Z\geq \gamma/3$ we conclude that the $r$th moment of $\bigtriangleup_{Z}$,
\begin{equation}\label{Delta-H}
\left\langle (\bigtriangleup_{Z})^r\right\rangle= H(r)+T 
\end{equation}
where, 
\begin{equation}\label{H(r)5.1}
H(r) := 
\sum\limits_{\substack{m_i\geq 1\\ 1\leq i\leq r}} 
\prod\limits_{i=1}\limits^{r}q^{-2m_{i}}m_{i}^{-1}
\sum\limits_{\substack{\deg f_{i}=m_{i}\\ 1\leq i \leq r\\f_{1}f_{2}...f_{r}=h^{2}}}
\Lambda(f_{1})\Lambda(f_{2})....\Lambda(f_{r})
\prod\limits_{P\mid h} (1- \mid P \mid ^{-1})(1+ \mid P \mid^{-2})
\end{equation}
and, $T=O(q^{-Z}).$
Taking $r^{\text{th}}$ power on both sides of \eqref{N_F5.1}we get 
\begin{eqnarray*}
	(N_{F_{t}})^{r} &=& (\bigtriangleup_{Z}(F_{t}))^{r} + \sum\limits_{l=1}\limits^{r}{r \choose l}(\bigtriangleup_{Z}(F_{t}))^{r-l}(\epsilon_{Z}(F_{t}))^{l}.
\end{eqnarray*}

Therefore
\begin{eqnarray*}
	\left\langle (N_{F_{t}})^{r}\right\rangle &=& \left\langle (\bigtriangleup_{Z}(F_{t}))^{r}\right\rangle +  \sum\limits_{l=1}\limits^{r}{r \choose l}\left\langle (\bigtriangleup_{Z}(F_{t}))^{r-l}(\epsilon_{Z}(F_{t}))^{l}\right\rangle. 
\end{eqnarray*} 
We use \eqref{upper-bounds5.1} in \eqref{main5.1} and \eqref{upper-bounds5.2} in \eqref{error5.1}, to get
\begin{eqnarray*}
	\sum\limits_{l=1}\limits^{r}{r \choose l}(\bigtriangleup_{Z}(F_{t}))^{r-l}(\epsilon_{Z}(F_{t}))^{l}
	&\ll & (\bigtriangleup_{Z}(F_{t}))^{r-1}(\epsilon_{Z}(F_{t}))\,\ll\, \frac{(\log{Z})^{r-1}}{q^{r-1}}\frac{g}{Z}q^{-3Z/2}. 
\end{eqnarray*} 
From \eqref{Delta-H} and the above estimate, we get 
\begin{equation*}
	\left\langle (N_{F_{t}})^{r}\right\rangle =H(r) + \bigcirc(q^{-Z})+ \bigcirc\left(\frac{g q^{-r-3Z/2 +1}}{Z}(\log Z)^{r-1}\right).
\end{equation*}

If $q$ is fixed then for each fixed positive integer $r$,
\begin{equation*}
\lim\limits_{\gamma\rightarrow \infty} \left\langle (N_{F_{t}})^{r}\right\rangle = H(r).
\end{equation*}

{\bf The distribution function:}\\

Suppose $\mathcal{R}$ is a random variable with the $r$-th moment
\begin{equation*}
\mathbb{E}(\mathcal{R}^{r}) = H(r) \qquad\,\,\forall r \in \mathbb{N}. 
\end{equation*}
The characteristic function of the random variable $\mathcal{R}$ is given by
\begin{equation*}
\phi(\tau)= \mathbb{E}(e^{i\tau\mathcal{R}})
= 1+ \sum\limits_{m=1}\limits^{\infty} \frac{(i\tau)^{m}}{m!}\mathbb{E}(\mathcal{R}^{m})
= 1+ \sum\limits_{m=1}\limits^{\infty} \frac{(i\tau)^{m}}{m!}H(m).
\end{equation*} 
The following result gives another representation of $H(r).$
\begin{prop}\label{H(r)1}
	For any positive integer $r \geq 1$, we have
	$$ H(r)= \sum_{m=1}^{r} \frac{r!}{2^{m}m!}
	\sum_{\substack{\sum_{i=1}^{t}\lambda_{i}=r \\ \lambda_{i}\geq 1}} 
	\sum_{\substack{P_{i}\text{distinct}\\ 1\leq i\leq m}}
	\prod_{i=1}^{m}\frac{u_{P_{i}}^{\lambda_{i}}+ (-1)^{\lambda_{i}}v_{P_{i}}^{\lambda_{i}}}
	{\lambda_{i}!(1+ \mid P_{i}\mid ^{-1})}$$
	where the sum on the right hand side is over all positive integer 
	$\lambda_{i}, i= 1, 2, ..., t$ such that $\sum\limits_{i=1}\limits^{m}\lambda_{i}=r$ 
	and over all distinct monic, irreducible polynomials $P_{i} \;\text{in}\; \F_{q}[x]$ with 
	\begin{eqnarray*}
		u_{P} &=& -\log (1-\mid P \mid^{-2}),\\
		v_{P} &=& \log(1+\mid P \mid^{-2})\quad \forall P\in \F_{q}[x].
	\end{eqnarray*}
\end{prop}
\begin{proof}
	This is essentially same as Proposition 1 of \cite{XiZa} with above definitions of 
	$ u_{P}$ and $ v_{P}$. Hence we skip the proof. 
\end{proof}

We know that characteristic function uniquely determines the distribution function. So  
applying Proposition \eqref{H(r)1} for $H(n)$ we find that
\begin{eqnarray*}  
	\phi(\tau)  &=& 
	1
	+
	 \sum_{n=1}^{\infty} \frac{(i\tau)^{n}}{n!}\sum_{r=1}^{n} \frac{n!}{2^{r}r!}
	\sum_{\substack{\sum_{j=1}^{r}\lambda_{j}=n\\ \lambda_{j}\geq 1}} 
	\sum_{\substack{P_{j}\text{distinct}\\1\leq j\leq r}}
	\prod_{j=1}\limits^{r}\frac{u_{P_{j}}^{\lambda_{j}}
		+ (-1)^{\lambda_{j}}v_{P_{j}}^{\lambda_{j}}}{\lambda_{j}!(1+ \mid P_{j}\mid ^{-1})}.
\end{eqnarray*}

Changing the order of summation we get,
\begin{eqnarray*}
	\phi(\tau)&=& 1
	+
	 \sum\limits_{r=1}\limits^{\infty}\frac{1}{2^{r}r!}
	\sum\limits_{\substack{P_{j}\text{distinct}\\1\leq j\leq r}}
	\prod\limits_{j=1}\limits^{r}
	\left(
	\sum\limits_{\lambda_{j}=1}\limits^{\infty}
	\frac{(i\tau)^{\lambda_{j}}
		\left(
		u_{P_{j}}^{\lambda_{j}}+ (-1)^{\lambda_{j}}v_{P_{j}}^{\lambda_{j}}
		\right)}
	{\lambda_{j}!(1+ \mid P_{j}\mid ^{-1})}
	\right)\\
	&=& 
	1+
	 \sum\limits_{r=1}\limits^{\infty}
	 \frac{1}{2^{r}r!}
	 \sum\limits_{\substack{P_{j}\text{distinct}\\1\leq j\leq r}}
	 \prod\limits_{j=1}\limits^{r}
	 \left(
	 \frac{(1-\mid P_{j} \mid^{-2})^{-i\tau}
	 	+ 
	 	(1+\mid P_{j} \mid^{-2})^{-i\tau}-2}{(1+ \mid P_{j}\mid ^{-1})}
	 \right).
\end{eqnarray*}

This completes the proof of $(1)$ of Theorem \ref{thm1.4}.

The next result gives an asymptotic formula for $H(r).$
\begin{prop}\label{H(r)2}
	For a fixed positive integer $r$,\\
	$\ds{H(r) = \dfrac{\delta_{r/2}r!}{2^{r/2}(r/2)!}q^{-3r/2} + \bigcirc_{r}(q^{-(r+1)3/2})}\;\text{as}\; q\rightarrow \infty.$\\ 
\end{prop}
\begin{proof}
	This is an exact analogue of Proposition 3 of \cite{XiZa} and we again skip the proof. 
\end{proof}
Now assume that both $\gamma,q \rightarrow \infty.$ For any fixed positive integer $r$, 
using Proposition \ref{H(r)2}, and \eqref{Delta-H} with the choice of $Z\geq \gamma/3,$ we get
\[
\left\langle (N_{F_{t}})^{r}\right\rangle 
= H(r) + \bigcirc(q^{-\gamma/3})
= \frac{\delta_{r/2}r!}{2^{r/2}(r/2)!}q^{-3r/2} 
+ \bigcirc_{r}(q^{-(r+1)3/2}+ q^{-\gamma/3}).
\]
Considering 
$q^{3/2}(\log\left(N_{q}\left( M_{L_{t}}(n,d)\right)  \right) - (n^2-1)(g-1)\log {q} )$ as a random variable on the space $\hh$, as both $\gamma,$ and $q \rightarrow \infty,$ we see that
all its moments are asymptotic to the corresponding moments of a 
standard Gaussian distribution where the odd moments vanish and the even moments are 
\begin{equation*}
\frac{1}{\sqrt{2\pi}}\int_{-\infty}^{\infty}\tau^{2r}e^{-\tau^{2}/2} d\tau = \frac{(2r)!}{2^{r}r!}.
\end{equation*} 
Hence the corresponding characteristic function converges to characteristic function of Gaussian distribution. Finally using Continuity theorem ( see Theorem $3.3.6$ in \cite{Du}), we obtain result $(2)$ of Theorem \ref{thm1.4}.

\section{Distribution on $M^{s}_{\mathcal{O}_{H_{t}}}(2,0)$}\label{distribution Ms}

In this section, all the asymptotic formulas we will be considering hold for the sufficiently large genus of hyperelliptic curves.
\subsection{Proof of Theorem \ref{thm1.5}}

Over the family of hyperelliptic curves $\hh,$ we recall from Proposition \ref{Ms2.1} that, the number of $\F_{q}$-rational points over the moduli space of stable vector bundles $M^{s}_{\mathcal{O}_{H_{t}}}(2,0)$ defined over the smooth projective hyperelliptic curve $H_{t}: y^{2}=F_{t}(x)$ in $\hh,$ is given by the following expression:
\begin{equation}\label{points over Ms}
N_{q}(M^{s}_{\mathcal{O}_{H_{t}}}(2,0))= q^{3g-3}\zeta_{H_{t}}(2)-\frac{\left( q^{g+1}-q^2+q\right) }{(q-1)^2(q+1)}N_{q}(J_{H_{t}})-\frac{1}{2(q+1)}N_{q^2}(J_{H_{t}})+ \frac{1}{2(q+1)}2^{2g}.
\end{equation}

For simplicity we denote,
\[
T_{1}:= q^{3g-3}\zeta_{H_{t}}(2)
\]
and
\[
T_{2}:= \dfrac{\left( q^{g+1}-q^2+q\right) }{(q-1)^2(q+1)}N_{q}(J_{H_{t}})- \dfrac{1}{2(q+1)}N_{q^2}(J_{H_{t}}) + \frac{1}{2(q+1)}2^{2g}
\]
Using Proposition \ref{log-zeta5.3} and Proposition \ref{log_Jac5.1} we see,
\[T_{2}=O(q^{2g}(\log{g})^{cN}),\]
for some absolute constant $c>0.$
Therefore,
\begin{equation*}
N_{q}(M^{s}_{\mathcal{O}_{H_{t}}}(2,0))
=T_{1}\left( 1+\frac{T_{2}}{T_{1}}\right) 
=q^{3g-3}\zeta_{H_{t}}(2)\left( 1+O\left( q^{-g}(\log{g})^{c'N} \right)\right) 
\end{equation*}
for some $c^{\prime}>c.$
Taking logarithm on both sides of the above equation, we have
\begin{equation}\label{eqn Ms2}
\log{N_{q}(M^{s}_{\mathcal{O}_{H_{t}}}(2,0))} - (3g-3)\log q =\log{\zeta_{H_{t}}(2)}+  O\left( q^{-g}(\log{g})^{c'N}\right).
\end{equation} 
We observe from equation \eqref{eqn Ms2} that, for $g \rightarrow\infty,$
Theorem \ref{thm1.5} follows by invoking Theorem $1.2$ in \cite{ASA}.

\section{Distribution on $\N$} \label{distribution N tilde}
\subsection{Proof of Theorem \ref{thm1.6}}

We recall from equation \eqref{pointN2.2} that over a smooth projective hyperelliptic curve $H_{t}: y^{2}=F_{t}(x)$ in $\hh,$
\begin{equation}\label{pointN}
N_{q}({\N}) \,=\, N_{q}(M^{s}_{\mathcal{O}_{H}}(2,0)) + N_{q}(Y) + 2^{2g}N_{q}(R)+ 2^{2g}N_{q}(S).
\end{equation}
Further, from \eqref{pointY} we write the $\F_{q}$-rational points on $Y$ as follows,
\begin{equation*}
N_{q}(Y) = \left(  \frac{q^{2g-3}-q}{2(q-1)(q+1)}\right)N_{q}(J_{H_{t}}) 
+\left(\frac{q^{2g-2}-1}{2(q-1)(q+1)} \right)N_{q^2}(J_{H_{t}})
-\left( \frac{q^{2g-3}-1}{2(q-1)}\right)2^{2g}. 
\end{equation*}
Also, we have
\begin{equation*}
N_{q}(G(2,g)) = \frac{(q^{g}-1)(q^{g-1}-1)}{(q-1)^{2}(q+1)},
\end{equation*}
and,
\begin{equation*}
N_{q}(G(3,g)) = \frac{(q^{g}-1)(q^{g-1}-1)(q^{g-2}-1)}{(q^3-1)(q-1)^{2}(q+1)}.
\end{equation*}
Therefore,
\begin{align*}
	N_{q}(G(2,g))+N_{q}(G(3,g)) 
	&=\frac{1}{(q^3-1)(q-1)^{2}(q+1)}\left\lbrace q^{3g-3}\left(1+O\left(\frac{1}{q^{g-5}}\right)\right)\right\rbrace \\
	&= q^{3g-9}\left( 1+ O\left( \frac{1}{q^{g-5}}\right) \right). 
\end{align*}
Using \eqref{points over Ms}, along with the above results in \eqref{pointN}, and further simplifying, we can write
\begin{equation}\label{eqn N}
N_{q}(\N)\,=\, U_{1}+U_{2},
\end{equation}
where
\[
U_{1}:=\frac{1}{2}q^{2g-4}\left( 1+O\left( \frac{1}{q}\right) \right)N_{q^2}(J_{H_{t}})
\]
and 
\[
U_{2}:=q^{3g-3}\log{\zeta_{H_{t}}(2)} + \frac{1}{2}q^{2g-5}\left( 1+O\left( \frac{1}{q}\right) \right) N_{q}(J_{H_{t}}) + q^{3g-9}\left( 1+O\left( \frac{1}{q}\right) \right)2^{2g}.
\]
Using Proposition \ref{log-zeta5.3} and Proposition \ref{log_Jac5.1}, we get 
\begin{equation*}
U_{2}= O\left( q^{\frac{7g}{2}}\right) 
\end{equation*}
and 
\begin{equation*}
U_{1}\geq q^{4g-4}\left( \log (7g)\right)^{-3}.
\end{equation*}
For large $q$ and $g$ we have ${\left( \frac{U_{2}}{U_{1}}\right) }=O(q^{-g/2}\left( \log(7g)\right)^3 ).$

Finally from \eqref{eqn N}, we write
\begin{equation*}
N_{q}({\N})=U_{1}\left( 1+O(q^{-g/2}\left( \log(7g)\right)^3)\right). 
\end{equation*}
Taking logarithm on both sides of the above equation and again using Proposition \ref{log_Jac5.1}, we find that
\begin{equation}\label{desingular}
\log{N_{q}({\N})}= \log{U_{1}}+O(q^{-g/2}\left( \log(7g)\right)^3)=(4g-4)\log q + O(1).
\end{equation} 
Next, we recall the following result of Xiong and Zaharescu (Theorem 2 of \cite{XiZa}).\\ 
\begin{lem}\label{XiZa2}
	\begin{enumerate}
		\item[(i)] If $q$ is fixed and $g\rightarrow\infty,$ then for $H_{t}\in \hh,$ the quantity \\
		$\log{N_{q}(J_{H_{t}})} -g\log{q}+\delta_{\gamma/2}\log(1- 1/q)$ converges weakly to a random variable $X,$ whose characteristic function $\phi(\tau)=\mathbb{E}(e^{i\tau X}) $ is given by 
		\begin{align*}
			\phi(\tau)&=
			 1
			 + 
			 \sum\limits_{r=1}\limits^{\infty}
			 \frac{1}{2^{r}r!}
			 \sum\limits_{\substack{P_{j}\text{distinct}\\1\leq j\leq r}}
			 \prod\limits_{j=1}\limits^{r}
			 \left(
			 \frac{(1-\mid P_{j} \mid^{-1})^{-i\tau}
			 	+ (1+\mid P_{j} \mid^{-1})^{-i\tau}-2}{(1+ \mid P_{j}\mid ^{-1})}
			 \right),   
		\end{align*}
		for all $\tau$ in $\mathbb{R}.$
		Here $P_{j}$'s are monic, irreducible polynomials in $\F_{q}[x]$ and $\mid P_{j}\mid= q^{\deg P_{j}}$. $ \delta_{\gamma/2}=1$ if $\gamma$ is even and $0$ otherwise.
		
		\item[(ii)] If both $q, g\rightarrow\infty,$ then for $H_{t}\in \hh,$ the quantity
		$q^{1/2}\left(\log{N_{q}(J_{H_{t}})} -g\log{q}\right)$ is distributed as a standard Gaussian.
		That is, for any $z\in \mathbb{R},$ we have, 
		\begin{equation*}
		\lim\limits_{\substack{q\rightarrow\infty\\g\rightarrow\infty}}
		\frac{1}{\#\hh}\#
		\left\lbrace H_{t} \in \hh:q^{1/2}
		\left(
		\log{N_{q}(J_{H_{t}})} -g\log{q}
		\right)
		\leq z 
		\right\rbrace 
		= \frac{1}{\sqrt{2\pi}}
		\int_{-\infty}^{z}e^{-\tau^{2}/2}d\tau.
		\end{equation*}
	\end{enumerate}
\end{lem}
We observe from \eqref{desingular} that, for $g \rightarrow\infty,$\\
\begin{equation*}
\log{N_{q}({\N})}-(4g-4)\log{q}=\log{N_{q^2}(J_{H_{t}})}-2g\log{q}.
\end{equation*}
So in Lemma \eqref{XiZa2}, if we change the base field from $F_{q}$ to $F_{q^2}$, we get that the quantity 
\[
\log{N_{q^2}(J_{H_{t}})} -2g\log{q}+\delta_{\gamma/2}\log(1- 1/q^2)
\]
converges weakly to a random variable $X,$ whose characteristic function $\phi(\tau)=\mathbb{E}(e^{i\tau X}) $ is as in the Lemma \eqref{XiZa2}(i), such that the monic, irreducible polynomials $P_{j}$s are in $\F_{q^2}[x],$ and so does the quantity 
\[
\log{N_{q}({\N})}-(4g-4)\log{q} + \delta_{\gamma/2}\log(1- 1/q^2),
\]
 and hence the proof of Theorem \eqref{thm1.6}$(1).$

		Next we see the case when both $q, g\rightarrow\infty.$ From Lemma \eqref{XiZa2}(ii), it follows that for $H_{t}:y^{2}=F_{t}(x)$ in $H_{\gamma, q^2},$ that is, for the the monic square-free polynomial $F_{t}(x)$ in $\F_{q^2}[x]$, the random variable
		\[
		q\left(\log{N_{q^2}(J_{H_{t}})} -2g\log{q}\right)
		\]
		 has a standard Gaussian distribution. Therefore as above we conclude that $q\left( \log{N_{q}({\N})}-(4g-4)\log{q} \right)$ is distributed as a standard Gaussian over the family $\hhh$, and hence the Theorem \eqref{thm1.6}$(2).$

\vspace{2cm}
\textbf{Acknowledgements:} We would like to thank Prof. D. S. Nagaraj for his valuable suggestions and comments when the results of this paper were presented to him. We also express our sincere gratitude to Prof. U.N. Bhosle for helping us in understanding the paper \cite{DeRa}, which was a great help while preparing this work. The second author also thanks Dr. Suhas B. N. for having many useful discussions, which helped her understand the various aspects of moduli spaces.

		\bibliographystyle{amsplain}
	\bibliography{mybibfile}

	\end{document}